\documentclass[11pt]{article}
\usepackage[dvips]{graphicx}
\usepackage[dvips,letterpaper,margin=1in]{geometry}
\usepackage{color}
\usepackage{amsmath}
\usepackage{amsfonts}
\usepackage{amssymb}
\usepackage{amsthm}
\usepackage{newlfont}
\usepackage{epsfig}
\usepackage{multirow}
\usepackage{setspace}
\usepackage{hyperref}
\usepackage{cite}

\begin{document}

\newtheorem{assumption}{Assumption}[section]
\newtheorem{definition}{Definition}[section]
\newtheorem{lemma}{Lemma}[section]
\newtheorem{proposition}{Proposition}[section]
\newtheorem{theorem}{Theorem}[section]
\newtheorem{corollary}{Corollary}[section]
\newtheorem{remark}{Remark}[section]

\small

\title{Translated Chemical Reaction Networks}
\author{Matthew D. Johnston \bigskip \\
\small Department of Mathematics\\
\small University of Wisconsin-Madison\\
\small 480 Lincoln Dr., Madison, WI 53706\\
\small email: mjohnston3@wisc.edu}
\date{}
\maketitle

\tableofcontents

\newpage

\begin{abstract}
\small
Many biochemical and industrial applications involve complicated networks of simultaneously occurring chemical reactions. Under the assumption of mass action kinetics, the dynamics of these chemical reaction networks are governed by systems of polynomial ordinary differential equations. The steady states of these mass action systems have been analysed via a variety of techniques, including elementary flux mode analysis, algebraic techniques (e.g. Groebner bases), and deficiency theory. In this paper, we present a novel method for characterizing the steady states of mass action systems. Our method explicitly links a network's capacity to permit a particular class of steady states, called toric steady states, to topological properties of a related network called a \emph{translated chemical reaction network}. These networks share their reaction stoichiometries with their source network but are permitted to have different complex stoichiometries and different network topologies. We apply the results to examples drawn from the biochemical literature.
\end{abstract}


\bigskip

\section{Introduction}


Chemical reaction networks are given by sets of reactions which react to form sets of products at a pre-determined kinetic rate. Under the simplest of kinetic assumptions, that of mass-action kinetics, we may model the dynamics of a continuously-mixed chemical process as an autonomous system of polynomial ordinary differential equations called a \emph{mass action system}. Despite the simplistic formulation of such systems, the resulting dynamical systems may exhibit a wide range of dynamical behaviors, including multistationarity \cite{C-F1,C-F2}, Hopf bifurcations \cite{W-H1,W-H2}, periodicity and chaos \cite{E-T}.

Particular attention has been given recently to the nature of the steady states of these mass action systems, and in particular to the \emph{positive} steady states (that is to say, steady states in $\mathbb{R}_{>0}^m$). Such analysis is complicated by two main factors (1) the non-linear nature of the steady state equations, and (2) the partitioning of the positive state space into invariant affine spaces called compatibility classes. The task of characterizing the steady state set of a mass action system is further complicated by the observation that, for applied chemical processes, many parameter values (i.e. the rate constants associated with each reaction) are typically unknown or only known to a certain precision; consequently, an emphasis has been placed on results which characterize the steady state set \emph{regardless} of the rate constant values.

Nevertheless, many general results about the steady states of mass action systems are well-known. It has been known since the 1970s that two foundation classes of mass action systems---\emph{detailed balanced systems} \cite{V-H} and \emph{complex balanced systems} \cite{H-J1}---possess a unique positive steady state within each positive compatibility class. These results were further related to the topological structure of the network's underlying reaction graph (reversibility and weak reversibility, respectively) in \cite{F1,H}. This network structure approach to characterizing steady states has been continued by Martin Feinberg in a series of papers focusing on network deficiency \cite{Fe2,Fe4,Fe3}, network injectivity \cite{C-F1,C-F2}, and concordance \cite{Sh-F2}. This author, together with Jian Deng, Christopher Jones, and Adrian Nachman, was also instrumental in producing a paper affirming the long-standing conjecture that every weakly reversible network contains a positive steady state \cite{D-F-J-N}.

Beginning with a series of papers published by Karin Gatermann in the early 2000s, interest arose for characterizing the steady state sets of mass action systems by using tools from algebraic geometry \cite{G,G-H,G-W}. Other prominent algebraists, including Alicia Dickenstein and Bernd Sturmfels, have since become involved in adapting chemical reaction network results and terminology to the algebraic setting. These authors, along with Gheorghe Craciun and Anne Shiu, were instrumental in making the connection between toric varieties, Birch's theorem from algebraic statistics, and complex balanced steady states in \cite{C-D-S-S}. This paved the way for the introduction of \emph{toric steady states}, a generalization of complex balanced steady states which no longer shared any direct correspondence on the topological structure of the reaction graph \cite{M-D-S-C}. Other related contributions to the study of the steady states of mass action systems have been made in \cite{Fe,Co,D-M,C1,M-H-K,C-F-R}.

Research on the steady states of chemical reaction systems has also be conducted for systems which do not possess traditional mass action kinetics. One recent example is that of \emph{generalized mass action systems} introduced by Stefen M\"{u}ller and Georg Regensburger \cite{M-R}. Generalized mass action systems maintain the topological structure of standard chemical reaction networks but allow the stoichiometry of the monomials appearing in the steady state conditions to differ from those inferred by the graphical structure. The authors show that a notion of complex balancing is maintained in this generalized setting and that steady state properties can often still be inferred from topological structure of the generalized reaction graph.

In this paper, we introduce a method for relating the steady states of a mass action system to those of a specially-constructed generalized mass action system. This method, called \emph{network translation}, will allow an explicit connection to be made between systems with toric steady states and generalized mass action systems with complex balanced steady states. It will also allow steady state properties to be inferred from generalized network parameters. As such, this paper can be seen as a step toward closing the gap between the network topology approaches to characterizing steady states championed by Martin Feinberg, \emph{et al.}, and the approaches of algebraists such as Karin Gatermann and Alicia Dickenstein. We apply the results to several well-studied networks contained in the biochemical literature.

While the primary application of this paper is characterizing the steady states of mass action system, it will be noted that \emph{translated chemical reaction networks} are interesting objects of study in their own right. We will close with a discussion of some avenues for future research, both within the study of translated chemical reaction networks and generalized chemical reaction networks in general.

\section{Background}

In this section, we present the terminology and notation relevant for the study of chemical reaction networks and mass action systems, which will be used throughout this paper. We will present these concepts both in the standard and generalized setting.

\subsection{Chemical Reaction Networks}

A chemical reaction network is given by a collection of elementary reactions of the form
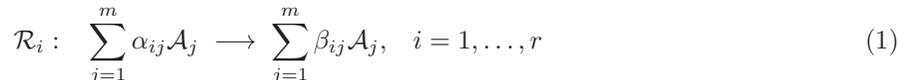
\begin{equation}
\label{75}
\mathcal{R}_i: \; \; \; \sum_{j=1}^m \alpha_{ij} \mathcal{A}_j \; \longrightarrow \; \sum_{j=1}^m \beta_{ij} \mathcal{A}_j, \; \; \; i = 1, \ldots, r
\end{equation}
where $\mathcal{S} = \left\{ \mathcal{A}_1, \ldots, \mathcal{A}_m \right\}$ is called the \emph{species set} and $\mathcal{R} = \left\{ \mathcal{R}_1, \ldots, \mathcal{R}_r \right\}$ is called the \emph{reaction set}. The coefficients $\alpha_{ij}, \beta_{ij} \in \mathbb{Z}_{\geq 0}^{r \times m}$ are called \emph{stoichiometric coefficients}. They control the number of individual molecules which are either consumed by, or produced as a result of, each individual reaction.

It is more common within chemical reaction network literature to index the reactions by the net terms on the left-hand or right-hand side of a reaction, which are called \emph{complexes}. In this setting, we remove redundancies so that, if a stoichiometrically equivalent complex appears multiply in the network (\ref{75}), it is only indexed once. We consequently define the \emph{complex set} to be $\mathcal{C} = \left\{  \mathcal{C}_1, \ldots, \mathcal{C}_n \right\}$ where each $\mathcal{C}_i = \sum_{j=1}^m y_{ij} \mathcal{A}_j$, $i=1,\ldots, n$, represents a stoichiometrically distinct complex and $y_i = (y_{i1}, \ldots, y_{im})$ is the \emph{stoichiometric vector} associated to the $i^{th}$ complex. The set of complexes which appear on the left (right) of at least one reaction are called \emph{reactant (product) complexes} and the \emph{reactant (product) complex set} is denoted $\mathcal{CR}$ ($\mathcal{CP}$). It is typically assumed that: (a) every species in $\mathcal{S}$ appears in at least one complex in $\mathcal{C}$; (b) every complex in $\mathcal{C}$ appears in at least one reaction in $\mathcal{R}$; and (c) there are no self-reactions (i.e. reactions $\mathcal{C}_i \rightarrow \mathcal{C}_i'$ where $\mathcal{C}_i = \mathcal{C}_i'$). A triplet $\mathcal{N} = (\mathcal{S},\mathcal{C},\mathcal{R})$ satisfying these conditions is called a \emph{chemical reaction network}.

In order to formalize the relationship between \emph{reaction-centered indexing} and \emph{complex-centered indexing}, we introduce the mappings $\rho: \mathcal{R} \mapsto \mathcal{CR}$ and $\rho': \mathcal{R} \mapsto \mathcal{CP}$. These mappings will be called the \emph{reactant profile} and \emph{product profile} of $\mathcal{N}$, respectively. They allow reactions to be represented in the condensed form $\mathcal{C}_{\rho(i)} \rightarrow \mathcal{C}_{\rho'(i)}$, $i=1, \ldots, r$. For example, consider the reaction network
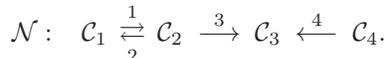
\[\mathcal{N}: \; \; \; \mathcal{C}_1 \; \mathop{\stackrel{1}{\rightleftarrows}}_{2} \; \mathcal{C}_2 \; \stackrel{3}{\longrightarrow} \; \mathcal{C}_3 \; \stackrel{4}{\longleftarrow} \; \mathcal{C}_4.\]
The reactant complex set is $\mathcal{CR} = \left\{ \mathcal{C}_1,\mathcal{C}_2,\mathcal{C}_4 \right\}$ and the reaction profile is $(\rho(\mathcal{R}_1),$$\rho(\mathcal{R}_2),$$\rho(\mathcal{R}_3),$$\rho(\mathcal{R}_4)) = (\mathcal{C}_1,$$\mathcal{C}_2,$$\mathcal{C}_2,$$\mathcal{C}_4)$. Correspondingly, the product complex set is $\mathcal{CP} = \left\{ \mathcal{C}_1,\mathcal{C}_2,\mathcal{C}_3 \right\}$ and the product profile is $(\rho'(\mathcal{R}_1),$$\rho'(\mathcal{R}_2),$$\rho'(\mathcal{R}_3),$$\rho'(\mathcal{R}_4) = (\mathcal{C}_2,$$\mathcal{C}_1,$$\mathcal{C}_3,$$\mathcal{C}_3)$. Notice that we assign independent indexings to the reactions and the complexes. This differs from much of chemical reaction network theory literature but will be a required feature of the analysis considered in this paper.

\begin{remark}
Throughout this paper, we will be more interested in the reactant profile of a chemical reaction network $\mathcal{N}$ than the product profile. This is because the reactant complexes control the monomials which appear in the corresponding kinetic systems.
\end{remark}

\begin{remark}[\textbf{Note on indexing}]

To avoid introducing excessive notation, when convenient we will allow the sets $\mathcal{C}$ and $\mathcal{R}$ to interchangeably denote the complexes/reactions themselves \emph{or} the corresponding index sets. For instance, we will allow $\mathcal{C}$ to denote $\left\{ \mathcal{C}_1, \ldots, \mathcal{C}_n \right\}$ \emph{or} the index set $\left\{ 1, \ldots, n \right\}$, depending on the context. We will also allow, for instance, $\rho(\mathcal{R}_i) = \mathcal{C}_j$ and $\rho(i)=j$ to interchangeably represent the correspondence of the $i^{th}$ reaction in $\mathcal{R}$ to the $j^{th}$ complex in $\mathcal{C}$. We will favor indicial notation whenever ambiguity is not a concern.
\end{remark}




\subsection{Reaction Graph}

Interpreting chemical reaction networks as interactions between stoichiometric-distinct complexes naturally gives rise to their interpretation as directed graphs $G(V,E)$ where the vertices are the complexes (i.e. $V = \mathcal{C}$) and the edges are the reactions (i.e. $E = \mathcal{R}$). In the literature this graph has been termed the \emph{reaction graph} of a network \cite{H-J1}.

There are several properties of a network's reaction graph of which we will need to be aware. Most importantly, we will need to be able to quantify the \emph{connections} between complexes. We will say that a complex $\mathcal{C}_i$ is \emph{connected to} $\mathcal{C}_j$ if these exists a sequence of complexes $\left\{ \mathcal{C}_{\mu(1)}, \ldots, \mathcal{C}_{\mu(l)} \right\}$ such that $\mathcal{C}_i = \mathcal{C}_{\mu(1)}$, $\mathcal{C}_j = \mathcal{C}_{\mu(l)}$, and either $\mathcal{C}_{\mu(k)} \rightarrow \mathcal{C}_{\mu(k+1)}$ or $\mathcal{C}_{\mu(k+1)} \rightarrow \mathcal{C}_{\mu(k)}$ for all $k = 1, \ldots, l-1$. Correspondingly, we will say that $\mathcal{C}_i$ is \emph{path-connected} to $\mathcal{C}_j$ if there is a sequence of complexes $\left\{ \mathcal{C}_{\mu(1)}, \ldots, \mathcal{C}_{\mu(l)} \right\}$ where all of the reactions are in the forward direction. A \emph{linkage class} is a maximally connected set of complexes and a \emph{strong linkage class} is a maximally path-connected set of complexes. That is to say, two complexes are in the same linkage class if and only if they are connected, and two complexes are in the same strong linkage class if and only if they are path-connected in both directions. For instance, consider the chemical reaction network
\begin{equation}
\label{38}
\mathcal{N}: \; \; \; \; \; \; \; \; \mathcal{C}_1 \rightarrow \mathcal{C}_2 \;\rightleftarrows\; \mathcal{C}_3 \; \; \; \; \; \; \; \; \mathcal{C}_4 \;\rightleftarrows\; \mathcal{C}_5.
\end{equation}
We have the linkage classes $\mathcal{L}_1 = \left\{ 1, 2, 3 \right\}$ and $\mathcal{L}_2 = \left\{ 4, 5 \right\}$, and the strong linkage classes $\mathcal{L}^s_1 = \left\{ 1 \right\}$, $\mathcal{L}^s_2 = \left\{ 2, 3 \right\}$ and $\mathcal{L}^s_3 = \left\{ 4, 5 \right\}$. It is worth noticing that linkage classes completely partition the complex set $\mathcal{C}$ of a network. That is to say, we have $\mathcal{L}_1 \cap \mathcal{L}_2 = \emptyset$ and $\mathcal{L}_1 \cup \mathcal{L}_2 = \mathcal{C}$. The number of linkage classes in a network will be denoted by $\ell$.


These concepts allow us to further classify chemical reaction networks. A chemical reaction network is said to be \emph{reversible} if $\mathcal{C}_i \to \mathcal{C}_j \in \mathcal{R}$ implies $\mathcal{C}_j \to \mathcal{C}_i \in \mathcal{R}$ and a chemical reaction network is said to be \emph{weakly reversible} if the linkage classes and strong linkage classes coincide. For example, we can see that the network (\ref{38}) is neither reversible nor weakly reversible since $\mathcal{C}_1 \to \mathcal{C}_2$ is in the network but there is no reaction $\mathcal{C}_2 \rightarrow \mathcal{C}_1$ or path leading from $\mathcal{C}_2$ to $\mathcal{C}_1$. To demonstrate how a network may be weakly reversible without being reversible, consider
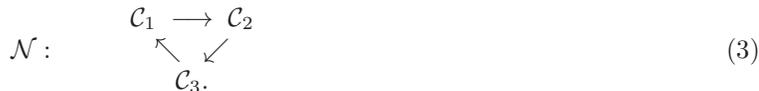
\begin{equation}
\label{39}
\mathcal{N}: \; \; \; \; \; \; \; \; \begin{array}{c} \mathcal{C}_1 \; \longrightarrow \; \mathcal{C}_2 \\ \nwarrow \; \; \; \swarrow \\ \mathcal{C}_3.\end{array}
\end{equation}
We have that $\mathcal{C}_1 \to \mathcal{C}_2$ is in (\ref{39}) but there is no reaction $\mathcal{C}_2 \rightarrow \mathcal{C}_1$ so that the network is not reversible; however, there is a path from $\mathcal{C}_2$ to $\mathcal{C}_1$ through the complex $\mathcal{C}_3$ so that the network is weakly reversible.


\begin{remark}
 We may easily extend our interpretation of a network's reaction graph to include a weighting $k_i$, $i =1, \ldots, r$, for each reaction. In this interpretation, we consider a \emph{weighted} directed graph $G(V,E)$ where each $\mathcal{R}_i \in E$ has weight $k_i$.
\end{remark}

\subsection{Mass Action Systems}

In order to model how the concentrations of the chemical species evolve over time, we assume that the reaction vessel is spatially homogeneous and that the reacting species are in sufficient quantity to be modeled as chemical concentrations. We will furthermore assume that the system obeys mass action kinetics, so that the rate of each reaction is proportional to the product of concentrations of the reactant species. That is to say, if the $i^{th}$ reaction has the form $\mathcal{A}_1 + \mathcal{A}_2 \rightarrow \cdots$ then we have rate $=k_i[\mathcal{A}_1][\mathcal{A}_2]$. The proportionality constant $k_i$ is commonly called the \emph{rate constant} of the reaction. The vector of rate constants will be denoted $k \in \mathbb{R}_{>0}^r$.

If we define $\mathbf{x} = (x_1, x_2, \ldots, x_m) \in \mathbb{R}_{>0}^m$ to be the vector of species concentrations, these assumptions give rise to the \emph{mass action system} $\mathcal{M} = (\mathcal{S},\mathcal{C},\mathcal{R},k)$ given by
\begin{equation}
\label{de}
\frac{d\mathbf{x}}{dt} = \sum_{i=1}^r k_i \left( y_{\rho'(i)} - y_{\rho(i)} \right) \mathbf{x}^{y_{\rho(i)}}
\end{equation}
where $y_{\rho(i)} = (y_{\rho(i)1}, y_{\rho(i)2}, \ldots, y_{\rho(i)m})$, $y_{\rho'(i)} = (y_{\rho'(i)1}, y_{\rho'(i)2}, \ldots, y_{\rho'(i)m})$, and $\mathbf{x}^{y_{\rho(i)}} = \prod_{j=1}^m x_j^{y_{\rho(i)j}}$. The vectors $y_{\rho'(i)}-y_{\rho(i)}$ are called \emph{reaction vectors}. They keep track of the net stoichiometric change in the individual species as the result of each reaction. They also give rise to the \emph{stoichiometric subspace} $S = \mbox{span}\left\{ y_{\rho'(i)} - y_{\rho(i)} \; | \; i = 1, \ldots, r \right\}\in \mathbb{R}^m$. The stoichiometric subspace partitions the state space $\mathbb{R}_{\geq 0}^m$ of (\ref{de}) into invariable affine spaces called \emph{stoichiometric compatibility classes}, $\mathsf{C}_{\mathbf{x}_0} = (S + \mathbf{x}_0) \cap \mathbb{R}_{>0}^m.$ Solutions of (\ref{de}) are restricted to $\mathsf{C}_{\mathbf{x}_0}$. That is to say, if $\mathbf{x}(t)$ by a solution of (\ref{de}) with $\mathbf{x}(0) = \mathbf{x}_0 \in \mathbb{R}_{>0}^m$, then $\mathbf{x}(t) \in \mathsf{C}_{\mathbf{x}_0}$ for all $t \geq 0$ \cite{H-J1,V-H}.

Most applications involving chemical reaction networks favor a matrix formulation of (\ref{de}) which explicitly separates the linear and nonlinear components of the equations. There are several such formulations in the literature which we will introduce in the relevant sections. We give here the most general decomposition of the network we will need. To that end, we introduce the following matrices:
\begin{itemize}
\item
The \emph{complex matrix} $Y \in \mathbb{Z}_{\geq 0}^{m \times n}$ is the matrix where the $j^{th}$ column is the $j^{th}$ stoichiometric vector $y_j$, i.e. $[Y]_{\cdot, j} = y_j$, $j=1, \ldots, n$.
\item
The matrix $I_a \in \mathbb{Z}_{\geq 0}^{n \times r}$ is the matrix with entries $[I_a]_{ji} = -1$ if $\rho(i)=j$, $[I_a]_{ji} = 1$ if $\rho'(i)=j$, and $[I_a]_{ji} = 0$ otherwise.
\item
The matrix $I_k \in \mathbb{R}_{\geq 0}^{r \times n}$ is the matrix with entries $[I_k]_{ij} = k_i$ if $\rho(i)=j$, and $[I_k]_{ij} = 0$ otherwise.
\end{itemize}
We will also need the \emph{mass action vector} $\Psi(\mathbf{x}) \in \mathbb{R}_{\geq 0}^n$, which is the vector with entries $\Psi_j(\mathbf{x}) = \mathbf{x}^{y_j}, j = 1, \ldots, n$. The following matrix form is equivalent to (\ref{de}):
\begin{equation}
\label{de2}
\frac{d\mathbf{x}}{dt} = Y \; I_a \; I_k \; \Psi(\mathbf{x}).
\end{equation}
The formulation (\ref{de2}) explicitly decouples the linearity associated with the network structure from the nonlinear mass action terms. It also explicitly relates the reaction-indexing scheme of the network to the complex-indexing scheme through the matrices $I_a$ and $I_k$. Consequently, this formulation provides a bridge between reaction-oriented kinetic approaches \cite{C1,A3,C-F-R} and complex-graph kinetic approaches \cite{H-J1,Fe2,Fe4}.





\subsection{Generalized Mass Action Systems}
\label{generalizedsection}

Many chemical reaction systems arising in practice do not obey the law of mass action. For this reason, and the desire to simplify models, it has become common in biochemical applications to model enzymatic reactions with alternative kinetics schemes, in particular with Michaelis-Menten kinetics \cite{M-M} or Hill kinetics \cite{H}.

Another alternative kinetic form is power-law formalism. In this formulation the kinetic terms are still monomials but they are permitted to take powers not necessarily corresponding to the stoichiometry of the reactant complex \cite{Sa}. This has recently been extended by Stefan M\"{u}ller and Georg Regensburger \cite{M-R} to a more network-focused approach called \emph{generalized chemical reaction networks}. 

\begin{definition}
A \textbf{generalized chemical reaction network} $(\mathcal{S},\mathcal{C},\tilde{\mathcal{C}},\mathcal{R})$ is a chemical reaction network $(\mathcal{S},\mathcal{C},\mathcal{R})$ together with a set of \textbf{kinetic complexes} $\tilde{\mathcal{C}}$ which are in a one-to-one correspondence with the elements of $\mathcal{C}$.
\end{definition}

\noindent The set $(\mathcal{S},\mathcal{C},\mathcal{R})$ determines the reaction structure and stoichiometry of the generalized chemical reaction network, just as it does for a standard chemical reaction network; however, each complex in $\mathcal{C}$ is associated to a kinetic complex in $\tilde{\mathcal{C}}$. These kinetic complexes $\tilde{\mathcal{C}}$ can be thought of as ``ghosting'' the chemical reaction network $(\mathcal{S}, \mathcal{C}, \mathcal{R})$ in that they do not appear directly in the reaction graph but are called upon when assigning a kinetics.

Since the kinetic complexes are in one-to-one correspondence with the stoichiometric complexes, we may consider properties of a second reaction graph with the kinetic complexes $\tilde{\mathcal{C}}$ in place of the regular complexes $\mathcal{C}$. This hypthetical \emph{kinetic reaction graph} does not determine the stoichiometry of the network but it does play an important role in determining where steady states of the corresponding kinetic model may lie. The \emph{kinetic-order subspace} $\tilde{S}$ as
\[\tilde{S} = \mbox{span} \left\{ \tilde{y}_{\rho'(i)} - \tilde{y}_{\rho(i)} \; | \; i = 1, \ldots, r \right\}\]
and the \emph{kinetic complex matrix} $\tilde{Y}$ as the matrix with the vectors $\tilde{y}_j$ corresponding to the kinetic complex $\tilde{\mathcal{C}}_j$ in the $j^{th}$ column. The stoichiometric and kinetic-order subspaces $S$ and $\tilde{S}$, respectively, are said to be \emph{sign-compatibility} if $\sigma(S) = \sigma(\tilde{S})$ where $\sigma(\cdot) \in \left\{ -, 0, + \right\}^m$ is the \emph{sign-vector}, i.e. the vector with entries $\sigma_i(\mathbf{v}) = $``$-$'' if $v_i <0$, $\sigma_i(\mathbf{v}) = 0$ if $v_i = $``$0$'', and $\sigma_i(\mathbf{v}) = $``$+$'' if $v_i >0$. (For details on properties of sign vectors, we defer to the concise introduction in \cite{M-R}.)

When space is not a concern, the ``ghosting'' of the complexes by kinetic complexes is denoted by dotted lines in the reaction graph. For example, we write
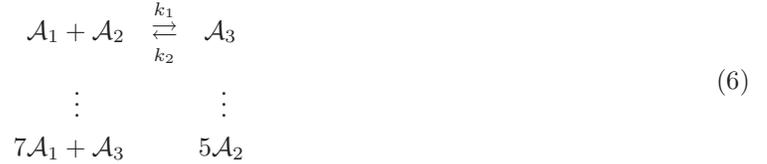
\begin{equation}
\label{network32}
\begin{split} \mathcal{A}_1 + \mathcal{A}_2 & \; \; \; \mathop{\stackrel{k_1}{\rightleftarrows}}_{k_2} \; \; \; \mathcal{A}_3 \\ \vdots \; \; \; \; \; \; & \; \; \; \; \; \; \; \; \; \; \; \; \; \vdots \\ 7 \mathcal{A}_1 + \mathcal{A}_3 & \; \; \; \; \; \; \; \; \; \; 5\mathcal{A}_2\end{split}
\end{equation}
to imply that the stoichiometric complex $\mathcal{C}_1 = \mathcal{A}_1 + \mathcal{A}_2$ is associated with the kinetic complex $\tilde{\mathcal{C}}_1 = 7 \mathcal{A}_1 + \mathcal{A}_3$ and that the stoichiometric complex $\mathcal{C}_2 = \mathcal{A}_3$ is associated with the kinetic complex $\tilde{\mathcal{C}}_2 = 5\mathcal{A}_2$. We can see immediately that $S = \mbox{ span} \left\{ (-1,-1,1) \right\}$ and $\tilde{S} = \mbox{ span} \left\{ (-7,5,-1) \right\}$ so that $\sigma(S) = \left\{ (-,-,+),\right.$ $(0,0,0),$ $\left.(+,+,-) \right\}$ and $\sigma(\tilde{S}) = \left\{ (-,+,-)\right.$ $,(0,0,0),$ $\left.(+,-,+) \right\}$ so that $S$ and $\tilde{S}$ are not sign-compatible.

The kinetic framework for generalized chemical reaction networks is the following.

\begin{definition}
The \textbf{generalized mass action system} $(\mathcal{S},\mathcal{C},\tilde{\mathcal{C}},\mathcal{R},k)$ corresponding to the generalized chemical reaction network $(\mathcal{S},\mathcal{C},\tilde{\mathcal{C}},\mathcal{R})$ is given by
\begin{equation}
\label{gde}
\frac{d\mathbf{x}}{dt} = Y \; I_a \; I_k \; \tilde{\Psi}(\mathbf{x})
\end{equation}
where $\tilde{\Psi}(\mathbf{x})$ has entries $\tilde{\Psi}_j(\mathbf{x}) = \mathbf{x}^{\tilde{y}_j}$, $j =1, \ldots, n$.
\end{definition}
\noindent In other words, a generalized mass action is the mass action system (\ref{de2}) with the monomials $\mathbf{x}^{y_j}$ replaced by the monomials $\mathbf{x}^{\tilde{y}_j}$. The generalized mass action system corresponding to network (\ref{network32}) is
\[\frac{dx_1}{dt} = \frac{dx_2}{dt} = -\frac{dx_3}{dt} = -k_1 x_1^7x_3 + k_2 x_2^5\]
where the stoichiometry of the network comes from the stoichiometric complexes $\mathcal{C}$ but the monomials come from the kinetic complexes $\tilde{\mathcal{C}}$.

\begin{remark}
It is worth noting that the support of the kinetic complexes $\tilde{\mathcal{C}}$ is not required to be the same as that of the original complex set $\mathcal{C}$ in the generalized chemical reaction network framework. This contrasts with some general kinetic frameworks for chemical reaction systems \cite{A3}.
\end{remark}

\section{Steady States of Mass Action Systems}



When considering the steady states of a mass action system, we are typically only interested in the \emph{positive steady state set} given by
\begin{equation}
\label{equilibrium}
E = \left\{ \mathbf{x} \in \mathbb{R}_{>0}^m \; | \; Y \; I_a \; I_k \; \Psi(\mathbf{x}) = 0 \right\}.
\end{equation}
Characterizing (\ref{equilibrium}) in general is a difficult algebraic task due to the nonlinear nature of the equations. It is somewhat surprising, therefore, that many characterizations exist within the literature which not only guarantee certain properties of the steady state set (\ref{equilibrium}), but guarantee these properties \emph{for all compatibility classes} and also \emph{for all rate constants}. In fact, there has been an emphasis on determining classes of mass action systems for which the steady state set is qualitatively identical regardless of which compatibility class or rate constants are chosen.

In this section, we will introduce three equilivalent reformulations of (\ref{equilibrium}) which have been used in the literature to characterize the steady states of mass action systems. This paper seeks to clarify the relationship between these three approaches.


\subsection{Stoichiometric and Cyclic Generators}
\label{generatorsection}

In order to derive a \emph{reaction-oriented} formulation of the steady state set (\ref{equilibrium}), we introduce the \emph{stoichiometric matrix} $\Gamma:= Y \; I_a \in \mathbb{Z}^{m \times r}$ and the \emph{reaction rate vector} $R(\mathbf{x}):= I_k \Psi(\mathbf{x}) \in \mathbb{R}_{\geq 0}^r$. The stoichiometric matrix is the matrix where the $i^{th}$ row is given by the $i^{th}$ reaction vector $y_{\rho'(i)}-y_{\rho(i)}$, while the reaction rate vector is the vector of kinetic forms for each reaction. These quantities allow us to rewrite (\ref{equilibrium}) as
\begin{equation}
 \label{equilibrium2}
E = \left\{ \mathbf{x} \in \mathbb{R}_{>0}^m \; | \; \Gamma \; R(\mathbf{x}) = \mathbf{0} \right\}.
\end{equation}
An immediate consequence of (\ref{equilibrium2}) is that, regardless of the chosen kinetics, a point $\mathbf{x} \in \mathbb{R}_{>0}^m$ is a steady state if and only if $R(\mathbf{x}) \in \mbox{ ker}(\Gamma)$. Consequently, a chemical reaction system may permit positive steady states only if ker$(\Gamma) \cap \mathbb{R}_{> 0}^r \not= \emptyset$. Characterization of the \emph{current cone} ker$(\Gamma) \cap \mathbb{R}_{ \geq 0}^r$ has formed the basis, explicitly and implicitly, of many classical results on the steady states of mass action systems \cite{C1,F1,H,G-W,C-F-R}.


Since the current cone is a finite-dimensional polyhedral cone, it can be finitely generated by a set of extreme vectors (Minkowski-Weyl theorem). Consequently, we have
\[\mbox{ker}(\Gamma) \cap \mathbb{R}_{\geq 0}^r = \sum_{i=1}^f \lambda_i E_i, \; \lambda_i \geq 0, \; \; \; \mbox{ where } \; \; \; E_i \in \mbox{ker}(\Gamma) \cap \mathbb{R}_{\geq 0}^m\]
where the $\left\{ E_1, E_2, \ldots, E_f \right\}$, are called the \emph{extreme currents} of the network. These currents represent modes of positive stoichiometric flux balance in the reaction network.

The extreme currents can be further subdivided by noticing that
\begin{equation}
\label{3294}
\mbox{dim(ker}(\Gamma)) = \mbox{dim(ker} (I_a)) + \mbox{dim}(\mbox{ker}(Y) \cap \mbox{Im}(I_a)).
\end{equation}
It follows that we may assign the generators $\left\{ E_1, \ldots, E_f \right\}$ of the current cone to one of two groups.
\begin{definition}
An extreme current $E_i$ of ker$(\Gamma) \cap \mathbb{R}_{\geq 0}^r$ is called:
\begin{enumerate}
\item a \textbf{cyclic generator} if $E_i \in$ ker$(I_a)$; or
\item a \textbf{stoichiometric generator} if $I_a E_i \in$ ker$(Y) \setminus \left\{ \mathbf{0} \right\}$.
\end{enumerate}
\end{definition}
\begin{remark}
Note that Karin Gatermann called cyclic generators \emph{positive circuits} in \cite{G-W}. The terminology is alterred here to emphasize the connection between the two sets as generators of ker$(\Gamma) \cap \mathbb{R}_{\geq 0}^m$.
\end{remark}

Although seldom explicitly stated, the distinction between stoichiometric and cyclic generators forms the basis of \emph{deficiency theory} introduced in \cite{F1,H} and studied extensively since \cite{F2,Fe2,Fe3,Fe4}.
\begin{definition}
\label{deficiency}
The \textbf{deficiency} of a chemical reaction network $\mathcal{N}$ is defined to be
\[\delta = \mbox{\emph{dim}}(\mbox{\emph{ker}}(Y) \cap \mbox{\emph{Im}}(I_a)).\]
\end{definition}

\noindent The deficiency is a nonnegative parameter which can be determined from the structure of the chemical reaction network itself. That is to say, it does not depend on the kinetic formulation, or even on the assumption of mass action kinetics. The deficiency of a generalized chemical reaction network $\tilde{\mathcal{N}} = ( \mathcal{S},\mathcal{C},\tilde{\mathcal{C}},\mathcal{R} )$ is defined in the same way as a regular chemical reaction network $\mathcal{N} = (\mathcal{S},\mathcal{C},\mathcal{R})$ by Definition \ref{deficiency}.

\begin{remark}
This definition of the deficiency differs from the classical definition, which is $\delta = n - \ell -s$ where $n$ is the number of stoichiometrically distinct complexes, $\ell$ is the number of linkage classes, and $s$ is the dimension of the stoichiometric subspace \cite{F1,H}. The definition given here is equivalent (see Appendix \ref{AppendixB}) and will be more intuitive for the operations introduced in Section \ref{techniquesection}.
\end{remark}

\begin{remark}
\label{remark}
 Notice by (\ref{3294}) that the condition $\delta = 0$ is equivalent to the network possessing no stoichiometric generators.

\end{remark}

\subsection{Complex Balanced Steady States}



In order to derive a \emph{complex-oriented} formulation of the steady state set (\ref{equilibrium}), we introduce the \emph{kinetic} (or \emph{Kirchhoff}) matrix defined by $A_k := I_a \; I_k \in \mathbb{R}^{n \times n}$. The kinetic matrix is closely related to the structure of the reaction graph of a network since $[A_k]_{ji} > 0$ for $i \not= j$ if and only if $\mathcal{C}_i \rightarrow \mathcal{C}_j$ is a reaction in the network. The set (\ref{equilibrium}) can be written in the equivalent form
\begin{equation}
 \label{equilibrium3}
E = \left\{ \mathbf{x} \in \mathbb{R}_{>0}^m \; | \; Y \; A_k \; \Psi(\mathbf{x}) = \mathbf{0} \right\}.
\end{equation}

An important class of steady states of mass action systems, derived from the form (\ref{equilibrium3}), are the \emph{complex balanced steady states}. This class was introduced by Fritz Horn and Roy Jackson in \cite{H-J1} as a generalization of \emph{detailed balanced steady states}.
\begin{definition}
\label{cb}
A positive steady state $\mathbf{x} \in \mathbb{R}_{>0}^m$ of a mass action system $\mathcal{M} = ( \mathcal{S},\mathcal{C},\mathcal{R},k)$ is called a \textbf{complex balanced steady state} if
\[\Psi(\mathbf{x}) \in \mbox{ \emph{ker}}(A_k).\]
Furthermore, a mass action system will be called a \textbf{complex balanced system} if every steady state is a complex balanced steady state.
\end{definition}
\noindent It is known that if a mass action system has a complex balanced steady state, then all steady states of complex balanced (Lemma 5B, \cite{H-J1}). Consequently, all mass action systems with complex balanced steady states are complex balanced systems. It is also known that every positive stoichiometric compatibility class $\mathsf{C}_{\mathbf{x}_0}$ of a complex balanced system contains precisely one steady state (Lemma 5A, \cite{H-J1}), and the complex balanced steady state set is given by
\[E = \left\{ \mathbf{x} \in \mathbb{R}_{>0}^m \; | \; \ln(\mathbf{x}) - \ln(\mathbf{a}) \in S^\perp \right\}\]
where $\mathbf{a} \in \mathbb{R}_{> 0}^m$ is an arbitrary complex balanced steady state of the system.


Further investigation of the complex balancing condition was conducted by Fritz Horn and Martin Feinberg in the papers \cite{F1,H}. Their main result, popularly called the \emph{Deficiency Zero Theorem}, relates the capacity of a network to permit complex balanced steady states to properties of the reaction graph.

\begin{theorem}[Theorem 4A, \cite{H}]
\label{deficiencyzerotheorem}
A mass action system $\mathcal{M} = (\mathcal{S},\mathcal{C},\mathcal{R},k)$ is complex balanced for all sets of rate constants if and only if the underlying reaction network is weakly reversible and the deficiency of the network is zero (i.e. $\delta = 0$).
\end{theorem}
\noindent This result gives computable properties, depending on the network topology alone, which are sufficient to guarantee strong restrictions on the nature, location, and number of steady states of the corresponding mass action systems. Remarkably, these results are guaranteed to hold for all possible choices of positive rate constants and all stoichiometric compatibility classes.


A surprising result of \cite{M-R} is that complex balancing may also be meaningfully defined for \emph{generalized} mass action systems.
\begin{definition}
\label{generalizedcb}
A positive steady state $\mathbf{x} \in \mathbb{R}_{>0}^m$ of a generalized mass action system $\tilde{\mathcal{M}} = ( \mathcal{S},\mathcal{C},\tilde{\mathcal{C}},\mathcal{R},k)$ is called a \textbf{generalized complex balanced steady state} if
\[\tilde{\Psi}(\mathbf{x}) \in \mbox{ \emph{ker}}(A_k).\]
\end{definition}
\noindent The authors show several important consequences of complex balancing. In particular, they show that a generalized chemical reaction network which permits generalized complex balanced steady states is weakly reversible (Proposition 2.18 \cite{M-R}) and that the steady state set is given by
\begin{equation}
\label{gcbequil}
E = \left\{ \mathbf{x} \in \mathbb{R}_{>0}^m \; | \; \ln(\mathbf{x}) - \ln(\mathbf{a}) \in \tilde{S}^\perp \right\}
\end{equation}
\noindent where $\mathbf{a} \in \mathbb{R}_{>0}^m$ is an arbitrary generalized complex balanced steady state of the system and $\tilde{S}$ is the kinetic-order subspace defined in Section \ref{generalizedsection}.

Furthermore, the authors define the kinetic deficiency by $\tilde{\delta} = \mbox{dim}(\mbox{ker}(\tilde{Y}) \cap \mbox{Im}(I_a))$. They show the following result.
\begin{theorem}[Proposition 2.20, \cite{M-R}]
\label{deficiencyzerotheorem}
A generalized mass action system $\tilde{\mathcal{M}} = (\mathcal{S},\mathcal{C},\tilde{\mathcal{C}},\mathcal{R},k)$ has at least one generalized complex balanced steady state for all sets of rate constants if the underlying reaction network is weakly reversible and the kinetic deficiency of the network is zero (i.e. $\tilde{\delta} = 0$).
\end{theorem}

\begin{remark}
It is important to note, however, that not all of the properties of standard complex balanced steady states apply in the generalized setting. In particular the generalized complex balanced steady state set $(\ref{gcbequil})$ may intersect a positive stoichiometric compatibility classes $\mathsf{C}_{\mathbf{x}_0}$ at a unique point, multiple times, or not at all.
\end{remark}

\subsection{Toric Steady States}


Karin Gatermann was instrumental in laying the groundwork for a transition in the study of steady states of mass action from a network topology setting to an algebraic one in the early 2000s \cite{G,G-H,G-W}. A number of prominent authors have since adopted this algebraic point of view; it may now be considered one of the principal approaches to determining steady state properties of mass action systems \cite{C-D-S-S,M-D-S-C,Shiu,C-F-R}. (We omit here background on the algebraic objects of interest, such as varieties, ideals, and Groebner bases. The interested reader is directed to the accessible textbook of Cox, Little, and O'Shea \cite{C-L-O}.)

In the algebraic geometry setting, the mass action steady state set (\ref{equilibrium}) is the \emph{variety} $V(I)$ associated with the \emph{mass action steady state ideal}
\[I =\langle Y \; I_a \; I_k \; \Psi(\mathbf{x}) \rangle.\]
That is to say, the mass action steady state ideal is the set of polynomials generated by the right-hand sides of (\ref{de2}). The question of characterizing the steady state set (\ref{equilibrium}) is equivalent to the algebraic question of characterizing the variety of strictly positive points, which is denoted $V_{>0}(I)$.

It was shown in \cite{C-D-S-S} that the steady state ideal for any complex balanced system is a \emph{toric} ideal. That is to say, it is a prime ideal which is generated by binomials. This justified the authors' choice to refashion complex balanced systems as \emph{toric dynamical systems}. The authors furthermore showed the following inclusion. (For a detailed introduction to the tree constants $K_i$ and $K_j$ (\ref{treeconstant}), and the relationship between Theorem \ref{theorem1} and Definition \ref{cb}, see Appendix \ref{AppendixC}.)

\begin{theorem}[Corollary 4, \cite{C-D-S-S}]
\label{theorem1}
The steady state ideal of a complex balanced system contains the binomials $K_i \mathbf{x}^{y_j} - K_j \mathbf{x}^{y_i}$ where $\mathcal{C}_i,\mathcal{C}_j \in \mathcal{L}_k$ and $K_i$ and $K_j$ are the tree constants of the $i^{th}$ and $j^{th}$ complex, respectively.
\end{theorem}

\noindent There are a number of desireable features which follow from a system having a toric ideal. It allows, for instance, an easy parametrization of the associated variety.

It was noted in \cite{M-D-S-C} that many mass action systems which do not admit complex balanced steady states nevertheless have steady state ideals which are generated by binomials. The authors say that such systems have \emph{toric steady states}. Although many properties of complex balanced systems do not generalize to systems with toric steady states, crucially, they do admit an easy parameterization of the positive variety $V_{>0}(I)$ (Theorem 3.11, \cite{M-D-S-C}).

In order to derive sufficient conditions for a system to have toric steady states, the authors of \cite{M-D-S-C} introduce the \emph{complex-to-species matrix} $\Sigma := Y \; I_a \; I_k \in \mathbb{R}^{m \times n}$ and rewrite (\ref{equilibrium}) as
\begin{equation}
 \label{equilibrium5}
E = \left\{ \mathbf{x} \in \mathbb{R}_{>0}^m \; | \; \Sigma \; \Psi(\mathbf{x}) = \mathbf{0} \right\}.
\end{equation}
From (\ref{equilibrium5}) it is easy to see that $\mathbf{x} \in \mathbb{R}_{>0}^m$ is a steady state of (\ref{equilibrium}) if and only if $\mathbf{x} \in \mbox{ker}(\Sigma)$. It is a surprising result of \cite{M-D-S-C} that, for many non-complex balanced systems, ker$(\Sigma)$ can in fact be decomposed in the same way as ker$(A_k)$ can be for complex balanced systems (see Appendix \ref{AppendixC}). The authors show the following inclusion.

\begin{theorem}[Theorem 3.3, \cite{M-D-S-C}]
\label{theorem012}
Consider a chemical reaction network $\mathcal{N} = (\mathcal{S},\mathcal{C},\mathcal{R})$. Suppose that ker$(\Sigma)$ has dimension $d$ and that there exists a partition $\Lambda_1, \Lambda_2, \ldots, \Lambda_d$ of $\left\{ 1, \ldots, n \right\}$ and a basis $\mathbf{b}_k$, $k=1, \ldots, d$, of ker$(\Sigma)$ with supp$(\mathbf{b}_k) = \Lambda_k$. Then the steady state ideal is generated by the binomials $[\mathbf{b}_k]_i \mathbf{x}^{y_j} - [\mathbf{b}_k]_j \mathbf{x}^{y_i}$ for $i,j \in \Lambda_k$, $k = 1, \ldots, d$.
\end{theorem}

It is striking that the binomials in Theorem \ref{theorem1} and Theorem \ref{theorem012} are both constructed by first partitioning the complexes of the chemical reaction network into disjoint components and then computing a \emph{basis} for a specific kernel restricted to the support of these components. Furthermore, the components in Theorem \ref{theorem1} have a clear interpretation in terms of the reaction network---they are the \emph{linkage classes} of the network's reaction graph. The components in Theorem \ref{theorem012} are less well-understood and are left as computational constructs in \cite{M-D-S-C}.

It is the clarification between the connection between Theorem \ref{theorem1} and Theorem \ref{theorem012}, and of complex balanced steady states and toric steady states in general, which will be the main concern of this paper. We will show that the supports of the components derived in Theorem \ref{theorem012} can often be corresponded to linkage classes just as they are in Theorem \ref{theorem1}. These linkage classes, however, will not be those of the original reaction network; rather, they will be the linkage classes of a related generalized reaction network which we will call a \emph{translated chemical reaction network}.

\section{Main Results}
\label{mainresultssection}

This section, we introduce the notion of \emph{network translation} and show how this concept can be used to characterize mass action systems with toric steady states.


\subsection{Translated Chemical Reaction Networks}
\label{section1}




The following is the foundational new concept of this paper.
\begin{definition}
\label{translation}
Consider a chemical reaction network $\mathcal{N} = (\mathcal{S}, \mathcal{C}, \mathcal{R})$ and a generalized chemical reaction network $\tilde{\mathcal{N}} = (\mathcal{S},\tilde{\mathcal{C}},\mathcal{CR}_K,\tilde{\mathcal{R}})$ where $\mathcal{CR}_K \subseteq \mathcal{CR}$. We will say $\tilde{\mathcal{N}}$ is a \textbf{translation} of $\mathcal{N}$ if:
\begin{enumerate}
\item
There is a bijection $h_1: \mathcal{R} \mapsto \tilde{\mathcal{R}}$ so that $ \tilde{y}_{\tilde{\rho}'(h_1(i))} - \tilde{y}_{\tilde{\rho}(h_1(i))} = y_{\rho'(i)} - y_{\rho(i)}$ for all $i=1, \ldots, r$;
\item
There is a surjection $h_2: \mathcal{CR} \mapsto \tilde{\mathcal{CR}}$ so that $h_2(\rho(i)) = \tilde{\rho}(h_1(i))$ for all $i=1, \ldots, r$; and
\item
The kinetic complex set $\mathcal{CR}_K \subseteq \mathcal{CR}$ of $\tilde{\mathcal{N}}$ contains exactly one element from the set $h_2^{-1}(j)$ for all $j \in \tilde{\mathcal{CR}}$, and this element is the kinetic complex associated with $j \in \tilde{\mathcal{CR}}$. The rest of the kinetic complexes (corresponding to strictly product complexes) may be drawn arbitrarily from $\mathcal{CR}_K$.
\end{enumerate}
The process of finding a generalized network $\tilde{\mathcal{N}}$ which is a translation of $\mathcal{N}$ will be called \textbf{network translation}.
\end{definition}

\begin{remark}
The labeling of the sets in the translation $\tilde{\mathcal{N}} = (\mathcal{S},\tilde{\mathcal{C}},\mathcal{CR}_K,\tilde{\mathcal{R}})$ differs from the standard definition of a generalized chemical reaction network in several very important ways. For a generalized chemical reaction network, $\tilde{\mathcal{C}}$ corresponded to the set of kinetic complexes, whereas for a translation it corresponds to the translated stoichiometric complexes of $\tilde{\mathcal{N}}$. We denote $\mathcal{CR}_K$ instead to correspond to the kinetic complexes of $\tilde{\mathcal{N}}$. Also note that the elements of the kinetic complex set $\mathcal{CR}_K$ are drawn from the reactant complex set $\mathcal{CR}$ of the \emph{original} network $\mathcal{N} = (\mathcal{S},\mathcal{C},\mathcal{R})$. To avoid confusion with the original network $\mathcal{N}$, we will also re-label the deficiencies of the translation $\tilde{\mathcal{N}}$ so that $\tilde{\delta}$ corresponds to the structural deficiency of $\tilde{\mathcal{N}}$ and $\tilde{\delta}_K$ corresponds to the kinetic deficiency of $\tilde{\mathcal{N}}$.
\end{remark}

\begin{remark}
The relationship between $\mathcal{R}$, $\tilde{\mathcal{R}}$, $\mathcal{CR}$, and $\tilde{\mathcal{CR}}$ through the mappings $h_1, h_2, \rho,$ and $\tilde{\rho}$ can be visualized as
\[\begin{array}{c} \displaystyle{\mathcal{R} \; \; \stackrel{h_1}{\longrightarrow} \; \; \tilde{\mathcal{R}}} \\ \displaystyle{\mathcal{N}: \hspace{1.25cm} \rho \downarrow \hspace{1.25cm} \downarrow \tilde{\rho} \hspace{1.25cm} :\tilde{\mathcal{N}}} \\ \displaystyle{\mathcal{CR} \; \stackrel{h_2}{\longrightarrow} \; \tilde{\mathcal{CR}},} \end{array}\]
This representation is useful when interpreting condition $2.$ of Definition \ref{translation}.
\end{remark}

\begin{remark}
Throughout the application portion of this paper, we will be \emph{constructing} translated networks and will therefore be able to control the indexing such that we will be able to take $h_1$ to be the identity mapping. In general, however, we permit the indexing of reactions to change in $\tilde{\mathcal{N}}$.
\end{remark}

The translated network $\tilde{\mathcal{N}}$ can be thought of as the network produced by translating the reactions of the original network $\mathcal{N}$ by adding or subtracting species to the left- and right-hand sides of each reaction, while preserving the original complexes as the kinetic complexes of the new (generalized) network. This operation conserves reactions, does not alter reaction vectors, and maps source complexes to source complexes. Up to reindexing, these conditions are the ones captured in the three requirements of Definition \ref{translation}.

There are several anomalous situations, however, which may arise from property $3.$ of Definition \ref{translation}. Notably, if multiple source complexes in $\mathcal{N}$ are mapped to the same source complex in $\tilde{\mathcal{N}}$, then the kinetic complex set $\mathcal{CR}_K$ is not necessarily uniquely defined. We therefore introduce the following strengthenings of network translation.

\begin{definition}
\label{supp}
Consider a chemical reaction network $\mathcal{N} = (\mathcal{S},\mathcal{C},\mathcal{R})$ and a translation $\tilde{\mathcal{N}} = (\mathcal{S},\tilde{\mathcal{C}},\mathcal{CR}_K,\tilde{\mathcal{R}})$. Then:
\begin{enumerate}
\item
We will say $\tilde{\mathcal{N}}$ is a \textbf{proper translation} of $\mathcal{N}$ if $h_2$ is injective as well as surjective.
\item
We will say $\tilde{\mathcal{N}}$ is a \textbf{strong translation} of $\mathcal{N}$ if $\tilde{\mathcal{N}}$ is weakly reversible.
\end{enumerate}
A translation $\tilde{\mathcal{N}}$ will be called \textbf{improper} if it is not proper.
\end{definition}


\begin{remark}
Proper translation removes the ambiguity in defining the kinetic complexes to source complexes in $\tilde{\mathcal{N}}$ (since $h^{-1}_2(j)$ is unique for each $j \in \tilde{\mathcal{CR}}$) while strong translation removes the ambiguity in defining the kinetic complexes to strictly product complexes in $\tilde{\mathcal{N}}$ (since there are none). Since every reactant complex of $\tilde{\mathcal{N}}$ appears as a kinetic complex for proper translations, we can define proper translations as $\tilde{\mathcal{N}} = (\mathcal{S},\tilde{\mathcal{C}},\mathcal{CR},\tilde{\mathcal{R}})$.
\end{remark}



For an example of how the network translation works, consider the standard Lotka-Volterra predator-prey system. The basic ecological interactions can be crudely modeled as the chemical reaction network $\mathcal{N}$ given by
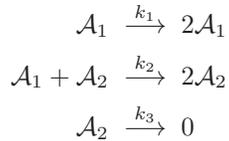
\[\begin{split} \mathcal{A}_1 \; & \stackrel{k_1}{\longrightarrow} \; 2\mathcal{A}_1 \\ \mathcal{A}_1 + \mathcal{A}_2 \; & \stackrel{k_2}{\longrightarrow} \; 2 \mathcal{A}_2 \\ \mathcal{A}_2 \; & \stackrel{k_3}{\longrightarrow} \; 0 \end{split}\]
where $\mathcal{A}_1$ corresponds to the prey and $\mathcal{A}_2$ corresponds to the predator.


Now consider the generalized chemical reaction network $\tilde{\mathcal{N}}_1$ given by
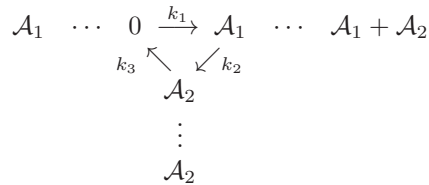
\[\begin{array}{c} \; \; \; \; \; \; \; \; \mathcal{A}_1 \; \; \; \cdots \; \; \; 0 \; \stackrel{k_1}{\longrightarrow} \; \mathcal{A}_1 \; \; \; \cdots \; \; \; \mathcal{A}_1 + \mathcal{A}_2\\ {}_{k_3} \nwarrow \; \; \; \swarrow_{k_2} \; \; \; \\ \mathcal{A}_2 \; \; \; \\ \vdots \; \; \; \\ \mathcal{A}_2 \; \; \; \end{array}\]
where the dotted lines correspond the complexes $\tilde{\mathcal{C}}_1 = 0$, $\tilde{\mathcal{C}}_2 = \mathcal{A}_1$, and $\tilde{\mathcal{C}}_3 = \mathcal{A}_2$ to the kinetic complexes $\mathcal{A}_1$, $\mathcal{A}_1+\mathcal{A}_2$, and $\mathcal{A}_2$ respectively. This network can be obtained from the original network by translating each reaction according to the following scheme:
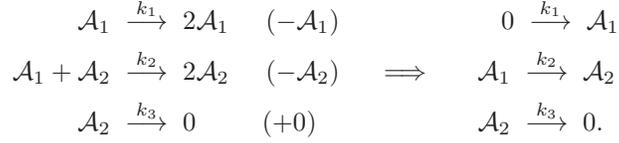
\[\begin{split} \mathcal{A}_1 \; & \stackrel{k_1}{\longrightarrow} \; 2\mathcal{A}_1 \; \; \; \; \; (- \mathcal{A}_1) \; \; \; \; \; \; \; \; \; \; \; \; \;\;\;\; \; \; \; \; \; \; 0 \; \stackrel{k_1}{\longrightarrow} \; \mathcal{A}_1 \\ \mathcal{A}_1 + \mathcal{A}_2 \; & \stackrel{k_2}{\longrightarrow} \; 2 \mathcal{A}_2 \; \; \; \; \; (- \mathcal{A}_2) \; \; \; \; \; \Longrightarrow \; \; \; \; \; \; \mathcal{A}_1 \; \stackrel{k_2}{\longrightarrow} \; \mathcal{A}_2\\ \mathcal{A}_2 \; & \stackrel{k_3}{\longrightarrow} \; 0 \; \; \; \; \; \; \; \; \; (+ 0) \; \; \; \; \; \; \; \; \; \; \; \; \; \; \; \; \; \; \; \; \; \; \mathcal{A}_2 \; \stackrel{k_3}{\longrightarrow} \; 0.\end{split}\]
It follows that the reactions are in a one-to-one correspondence and that the associated reaction vectors are the same. Furthermore, since the mapping from the source complexes is unique and the network is weakly reversible, we have that $\tilde{\mathcal{N}}_1 = (\mathcal{S},\tilde{\mathcal{C}},\mathcal{CR},\tilde{\mathcal{R}})$ is a proper and strong translation of $\mathcal{N}$.

It is interesting to note that translations are not unique, even translations which are strong and proper. We could have, for instance, chosen the reaction translations 
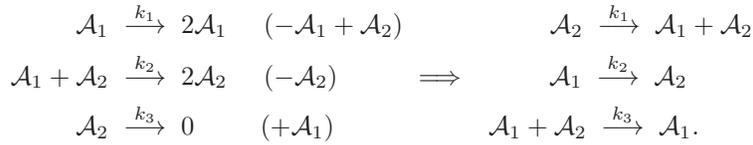
\[\begin{split} \mathcal{A}_1 \; & \stackrel{k_1}{\longrightarrow} \; 2\mathcal{A}_1 \; \; \; \; \; (-\mathcal{A}_1 + \mathcal{A}_2) \; \; \; \; \; \; \; \; \; \; \; \; \; \; \; \; \; \; \;\; \mathcal{A}_2 \; \stackrel{k_1}{\longrightarrow} \; \mathcal{A}_1 + \mathcal{A}_2 \\ \mathcal{A}_1  + \mathcal{A}_2 \; & \stackrel{k_2}{\longrightarrow} \; 2 \mathcal{A}_2 \; \; \; \; \; (- \mathcal{A}_2) \; \; \; \; \; \; \; \; \; \; \Longrightarrow \; \; \; \; \; \; \; \; \; \; \; \mathcal{A}_1 \; \stackrel{k_2}{\longrightarrow} \; \mathcal{A}_2\\ \mathcal{A}_2 \; & \stackrel{k_3}{\longrightarrow} \; 0 \; \; \; \; \; \; \; \; \; (+ \mathcal{A}_1) \; \; \; \; \; \; \; \; \; \; \; \; \; \; \; \; \; \; \; \;\; \mathcal{A}_1 + \mathcal{A}_2 \; \stackrel{k_3}{\longrightarrow} \; \mathcal{A}_1.\end{split}\]
This gives the strongly translated chemical reaction network $\tilde{\mathcal{N}}_2$
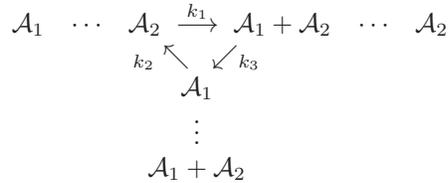
\[\begin{array}{c} \; \; \; \; \; \; \mathcal{A}_1 \; \; \; \cdots \; \; \; \mathcal{A}_2 \; \stackrel{k_1}{\longrightarrow} \; \mathcal{A}_1+\mathcal{A}_2 \; \; \; \cdots \; \; \; \mathcal{A}_2\\ \; \; \; \; \; \; {}_{k_2} \nwarrow \; \; \; \swarrow_{k_3} \; \; \; \; \; \; \; \; \; \\ \; \; \; \; \; \; \mathcal{A}_1  \; \; \; \; \; \; \; \; \; \\ \; \; \; \; \; \vdots \; \; \; \; \; \; \; \; \\ \mathcal{A}_1 + \mathcal{A}_2 \; \; \; \end{array}\]
This translation is also proper, but has the reaction cycle oriented in the opposite direction.

\subsection{Properly Translated Mass Action Systems}
\label{section2}


We assign a kinetics to a proper translation in the following way.
\begin{definition}
\label{properkinetic}
Suppose $\tilde{\mathcal{N}} = (\mathcal{S},\tilde{\mathcal{C}},\mathcal{CR},\tilde{\mathcal{R}})$ is a proper translation of a chemical reaction network $\mathcal{N} = (\mathcal{S},\mathcal{C},\mathcal{R})$ and $\mathcal{M} = (\mathcal{S},\mathcal{C},\mathcal{R},k)$ is a mass action system corresponding to $\mathcal{N}$. Then we define the \textbf{properly translated mass action system} of $\mathcal{M}$ to be the generalized mass action system $\tilde{\mathcal{M}} = (\mathcal{S},\tilde{\mathcal{C}},\mathcal{CR},\tilde{\mathcal{R}},\tilde{k})$ where $\tilde{k}_{h_1(i)} = k_i$. 
\end{definition}

\begin{remark}
This relationship is the natural correspondence since we make the same correspondence for rate constants as we make for reactions. In other words, for proper translations, we will simply transfer the rate constant along with the reaction in the translation process. The correspondence for \emph{improper} translations will be more complicated, if we can make a sensible correspondence at all (see Definition \ref{improperkinetic}).
\end{remark}

The following relates the dynamics of a properly translated mass action system $\tilde{\mathcal{M}}$ to the original mass action system $\mathcal{M}$.

\begin{lemma}
\label{lemma231}
Suppose $\tilde{\mathcal{M}} = (\mathcal{S},\tilde{\mathcal{C}},\mathcal{CR}_K,\tilde{\mathcal{R}},\tilde{k})$ is a properly translated mass action system of $\mathcal{M} = (\mathcal{S},\mathcal{C},\mathcal{R},k)$. Then the generalized mass ation system (\ref{gde}) governing $\tilde{\mathcal{M}}$ is identical to the mass action system (\ref{de}) governing $\mathcal{M}$. In particular, the two systems have the same steady states.
\end{lemma}

\begin{proof}
Consider a chemical reaction network $\mathcal{N} = (\mathcal{S},\mathcal{C},\mathcal{R})$ with corresponding mass action system $\mathcal{M} = (\mathcal{S},\mathcal{C},\mathcal{R},k)$ and a proper translation $\tilde{\mathcal{N}} = (\mathcal{S},\tilde{\mathcal{C}},\mathcal{CR},\tilde{\mathcal{R}})$ of $\mathcal{N}$. Let $\tilde{\mathcal{M}} = (\mathcal{S},\tilde{\mathcal{C}},\mathcal{CR},\tilde{\mathcal{R}},\tilde{k})$ be a properly translated mass action system of $\mathcal{M} = (\mathcal{S},\mathcal{C},\mathcal{R},k)$ defined by Definition \ref{properkinetic}. Without loss of generality, we will index the reactions of $\tilde{\mathcal{N}}$ so that $h_1$ may be taken to be the identity.

Since $\tilde{\mathcal{N}}$ is a translation of $\mathcal{N}$, it follows from property $1.$ of Definition \ref{translation} that the system (\ref{de2}) governing $\mathcal{M}$ is given by
\[\displaystyle{\frac{d\mathbf{x}}{dt}} = Y \; I_a \; I_k \; \Psi(\mathbf{x}) = \tilde{Y} \; \tilde{I}_a \; I_k \; \Psi(\mathbf{x})\]
where $\tilde{Y}$ and $\tilde{I}_a$ correspond to the translation $\tilde{\mathcal{N}}$. It remains to relate the rate vector $R(\mathbf{x}) := I_k \; \Psi(\mathbf{x})$ to $\tilde{R}(\mathbf{x}) := \tilde{I}_k \; \Psi_K(\mathbf{x})$ corresponding to $\tilde{\mathcal{M}}$. Notice that we have:
\begin{itemize}
\item $[\Psi_K(\mathbf{x})]_j = \mathbf{x}^{y_{h_2^{-1}(j)}}$ for all $j \in \tilde{\mathcal{CR}}$ by Property $3.$ of Definition \ref{translation};
\item $\tilde{k}_i = k_i$, $i=1, \ldots, r$, by Definition \ref{properkinetic};
\item $h_2(\rho(i)) = \tilde{\rho}(i)$ for all $i=1, \ldots, r$, by Property $2.$ of Definition \ref{translation}.
\item $h_2^{-1}(h_2(j)) = j$ for all $j \in \mathcal{CR}$ by Condition $1.$ of Definition \ref{supp}.
\end{itemize}
It follows that, for all $i = 1, \ldots, r$, $\tilde{R}_i(\mathbf{x}) = \tilde{k}_i \mathbf{x}^{y_{h_2^{-1}(\tilde{\rho}(i))}} = k_i \mathbf{x}^{y_{h_2^{-1}(\tilde{\rho}(i))}}= k_i \mathbf{x}^{y_{h_2^{-1}(h_2(\rho(i)))}}= k_i \mathbf{x}^{y_{\rho(i)}} = R_i(\mathbf{x})$. Consequently, we have
\begin{equation}
 \label{382}
\frac{d\mathbf{x}}{dt} = \tilde{Y} \; \tilde{I}_a \; I_k \; \Psi(\mathbf{x}) = \tilde{Y} \; \tilde{I}_a \; \tilde{I}_k \; \Psi_K(\mathbf{x})
\end{equation}
so that $\mathcal{M}$ and $\tilde{\mathcal{M}}$ have the same dynamics, and we are done.

\end{proof}

\begin{remark}
This result says that, for proper translations, transferring the rate constants along with the reaction arrows produces a generalized mass action system with the same dynamics as the original mass action system. The hope is that the steady states of the generalized mass action system produced by this translation can be more easily characterized than the steady states of the original mass action system can be through direct analysis. In Section \ref{section4} we will show how this intuition can be used to characterized mass action systems with toric steady states.
\end{remark}

\subsection{Improperly Translated Mass Action Systems}
\label{section3}




Sensibly defining a generalized mass action system $\tilde{\mathcal{M}} = (\mathcal{S},\tilde{\mathcal{C}},\mathcal{CR}_K,\mathcal{R},\tilde{k})$ for an improper translation $\tilde{\mathcal{N}} = (\mathcal{S},\tilde{\mathcal{C}},\mathcal{CR}_K,\mathcal{R})$ is more challenging than for proper translations because improper translations \emph{do not conserve} source complexes. That is to say, there is at least one source complex in $\mathcal{N}$ which does not appear as the kinetic complex of any source complex in $\tilde{\mathcal{N}}$. As a result, an improperly translated generalized mass action system $\tilde{\mathcal{M}}$ will necessarily have fewer monomials than the original mass action system $\mathcal{M}$, and a comprehensive dynamical result of the form Lemma \ref{lemma231} will not be possible.


In this section, we introduce conditions---which we call resolvability conditions---under which the \emph{steady state set} of a generalized mass action system $\tilde{\mathcal{M}}$ corresponding to an improper translation $\tilde{\mathcal{N}}$ can be shown to coincide with that of the original mass action system $\mathcal{M}$. The trick to making this correspondence will be in relating the source complexes in $\mathcal{CR} \setminus \mathcal{CR}_K$ to those in $\mathcal{CR}_K$. This will allow us to define the rate constants $\tilde{k}_i$ of the generalized mass action system $\tilde{\mathcal{M}}$ corresponding to $i \in \mathcal{R}$ such that $\rho(i) \in \mathcal{CR} \setminus \mathcal{CR}_K$ in such a way as to preserve the steady state set.

We start by giving the following definitions.

\begin{definition}
Suppose $\tilde{\mathcal{N}} = (\mathcal{S},\tilde{\mathcal{C}},\mathcal{CR}_K,\tilde{\mathcal{R}})$ is an improper translation of a chemical reaction network $\mathcal{N} = (\mathcal{S},\mathcal{C},\mathcal{R})$. Then:
\begin{enumerate}
\item
A translated source complex $\tilde{\mathcal{C}}_j \in \tilde{\mathcal{CR}}$ will be called an \textbf{improper} or \textbf{conflicted complex} if there exists $p, q \in h^{-1}_2(j)$ such that $p \not= q$. The set of improper complexes will be denoted $\tilde{\mathcal{C}}_I \subseteq \tilde{\mathcal{C}}$.
\item
A reaction $\mathcal{R}_i \in \mathcal{R}$ will be called an \textbf{improper} or \textbf{conflicted reaction} if $\rho(i) \in \mathcal{CR} \setminus \mathcal{CR}_K$. The set of improper reactions will be denoted $\mathcal{R}_I \subseteq \mathcal{R}$.
\item
A source complex $\mathcal{C}_j \in \mathcal{CR}_K$ will be called \textbf{kinetically relevant} to the reaction $\mathcal{R}_i \in \mathcal{R}_I$ if $h_2^{-1}(h_2(\rho(i))) \cap \mathcal{CR}_K = j$. The index of the \textbf{kinetically relevant complex} of $\mathcal{R}_i$ will be denoted $\rho(i)_K = h_2^{-1}(h_2(\rho(i))) \cap \mathcal{CR}_K$.
\end{enumerate}
\end{definition}

\begin{remark}
These definitions clarify some the objects which are unique to improper translations $\tilde{\mathcal{N}}$. A translated complex will be called an \emph{improper complex} if multiple source complexes in $\mathcal{CR}$ are translated to it and a reaction will be called an \emph{improper reaction} if it is assigned a different kinetic complex in the translation $\tilde{\mathcal{N}}$ than its own source complex in $\mathcal{N}$. Finally, the index of the kinetically relevant complex of the $i^{th}$ improper reaction is denoted $\rho(i)_K$.
\end{remark}

\begin{remark}
Notice that for all $\mathcal{R}_i \in \mathcal{R} \setminus \mathcal{R}_I$ we have $\rho(i)_K = \rho(i)$. That is to say, the kinetic relevant complex corresponding to a proper reaction coincides with the pre-translation source complex.
\end{remark}

We now want to explicitly relate the complexes in $\mathcal{CR} \setminus \mathcal{CR}_K$ to those in $\mathcal{CR}_K$. In order to accomplish this, we introduce the following subspace and weak resolvability condition for improper translations.

\begin{definition}
\label{improperkineticsubspace}
Suppose $\tilde{\mathcal{N}} = (\mathcal{S},\tilde{\mathcal{C}},\mathcal{CR}_K,\tilde{\mathcal{R}})$ is an improper translation of a chemical reaction network $\mathcal{N} = (\mathcal{S},\mathcal{C},\mathcal{R})$. We define the \textbf{improper kinetic subspace} $\tilde{S}_I$ of $\tilde{\mathcal{N}}$ to be
\[\tilde{S}_{I} = \mbox{span} \left\{ \bigcup_{i \in \mathcal{R}_I} \left\{ y_{\rho(i)} - y_{\rho(i)_K} \right\} \right\}.\]
We will say that an improper translation $\tilde{\mathcal{N}}= (\mathcal{S},\tilde{\mathcal{C}},\mathcal{CR}_K,\tilde{\mathcal{R}})$ of a chemical reaction network $\mathcal{N} = (\mathcal{S},\mathcal{C},\mathcal{R})$ is \textbf{weakly resolvable} if it is strong and if $\tilde{S}_{I} \subseteq \tilde{S}$.
\end{definition}


The improper kinetic subspace $\tilde{S}_{I}$ is the space given by the span of stoichiometric differences of complexes mapped to an improper complex. The primary consequence of weak resolvability is that it explicitly relates the monomials corresponding to the source complexes of improper reactions to the monomial corresponding to the kinetically relevant complex of the reaction. This is the content of the following result.

\begin{lemma}
\label{lemma00}
Suppose $\tilde{\mathcal{N}} = (\mathcal{S},\tilde{\mathcal{C}},\mathcal{CR}_K,\tilde{\mathcal{R}})$ is a weakly resolvable improper translation of a chemical reaction network $\mathcal{N} = (\mathcal{S},\mathcal{C},\mathcal{R})$. Let $\left\{ y_{p_j}-y_{q_j} \right\}_{j=1}^{\tilde{s}}$ where $p_j, q_j \in \mathcal{CR}$ denote any subset of the pairs in (\ref{tildeS}) which forms a basis of $\tilde{S}$. Then, for every $\mathcal{R}_i \in \mathcal{R}_I$ there exist constants $c_i$, $i=1, \ldots, \tilde{s}$, such that
\begin{equation}
 \label{term}
\mathbf{x}^{y_{\rho(i)}} = \left[ \tilde{K}_{\rho(i),\rho(i)_K}(\mathbf{x}) \right] \; \mathbf{x}^{y_{\rho(i)_K}}
\end{equation}
where the \textbf{weak kinetic adjustment factor} of $\rho(i)$ and $\rho(i)_K$, $\tilde{K}_{\rho(i),\rho(i)_K}(\mathbf{x})$, is given by
\begin{equation}
\label{adjustment}
\tilde{K}_{\rho(i),\rho(i)_K}(\mathbf{x}) = \prod_{i=1}^{\tilde{s}} \left( \frac{\mathbf{x}^{y_{p_i}}}{\mathbf{x}^{y_{q_i}}} \right)^{c_i}.
\end{equation}
\end{lemma}

\begin{proof}
Consider an arbitrary $\mathcal{R}_i \in \mathcal{R}_I$ and the difference $y_{\rho(i)} - y_{\rho(i)_K}$. Define a basis for $\tilde{S}$ by $\left\{ y_{p_i}-y_{q_i} \right\}_{i=1}^{\tilde{s}}$ where $p_i, q_i \in \mathcal{CR}$ by removing linearly dependent vectors from the generating set (\ref{tildeS}). Since $\tilde{S}_I \subseteq \tilde{S}$, it follows that we can write
\begin{equation}
\label{fkl}
y_{\rho(i)} - y_{\rho(i)_K} = \sum_{i=1}^{\tilde{s}} c_i ( y_{p_i} - y_{q_i} )
\end{equation}
where $c_1, \ldots, c_{\tilde{s}}$, are constants. The form (\ref{term}) follows directly by rearranging (\ref{fkl}) and raising the terms into the exponent of $\mathbf{x} \in \mathbb{R}^m_{>0}$, and we are done.
\end{proof}

The identity (\ref{term}) gives an explicit relationship between the monomials $\mathbf{x}^{y_{\rho(i)_K}}$ corresponding to kinetically relevant complexes and the monomials $\mathbf{x}^{y_{\rho(i)}}$ corresponding to complexes which are not kinetically relevant. We notice, however, that the weak kinetic adjustment factor (\ref{adjustment}) depends explicitly on $\mathbf{x} \in \mathbb{R}_{>0}^m$. In order to define conditions which remove this dependence, we must first introduce the following reaction graph.


\begin{definition}
\label{hypothetical}
Suppose $\tilde{\mathcal{N}} = (\mathcal{S},\tilde{\mathcal{C}},\mathcal{CR}_K,\tilde{\mathcal{R}})$ is an improper translation of a chemical reaction network $\mathcal{N} = (\mathcal{S},\mathcal{C},\mathcal{R})$ and $\mathcal{M} = (\mathcal{S},\mathcal{C},\mathcal{R},k)$ is a mass action system corresponding to $\mathcal{N}$. We define the \textbf{semi-proper reaction graph} of $\mathcal{\tilde{N}}$ to be the weighted reaction graph $\tilde{G}(\tilde{V},\tilde{E})$ with $\tilde{V} = \tilde{\mathcal{C}}$, $\tilde{E} = \tilde{\mathcal{R}}$, and edge weights given by
\[\tilde{E}_{h_1(i)} = \left\{ \begin{array}{ll} k_i, \; \; \; & \mbox{for } i \in \mathcal{R} \setminus \mathcal{R}_I \\ \tilde{k}_i, \; \; \; \; \; \; \; \; & \mbox{for } i \in \mathcal{R}_I \end{array} \right.\]
where $k_i$ are the rate constants of $\mathcal{M}$ and $\tilde{k}_i$ are undetermined positive constants.
\end{definition}


The semi-proper reaction graph is obtained by making the natural choice for the rate constants in the translation $\tilde{\mathcal{N}}$ for \emph{proper} reactions (i.e. the choice we made in Definition \ref{properkinetic}) while leaving the rest of the rate constants undetermined. In order to sensibly define the rate constants for improper reactions, we introduce the following strengthening of the earlier resolvability condition.




\begin{definition}
\label{stronglyresolvable}
Let $\tilde{\mathcal{N}} = (\mathcal{S},\tilde{\mathcal{C}},\mathcal{CR}_K,\tilde{\mathcal{R}})$ be a weakly resolvable improper translation of a chemical reaction network $\mathcal{N} = (\mathcal{S},\mathcal{C},\mathcal{R})$ and $\tilde{G}(\tilde{V},\tilde{E})$ denote the semi-proper reaction graph of $\tilde{\mathcal{N}}$. For $i \in \mathcal{R}_I$, we define the \textbf{strong kinetic adjustment factor} of $\rho(i)$ and $\rho(i)_K$ to be
\begin{equation}
\label{strongkineticadjustmentfactor}
\tilde{K}_{\rho(i),\rho(i)_K} = \prod_{i=1}^{\tilde{s}} \left( \frac{\tilde{K}_{h_2(p_i)}}{\tilde{K}_{h_2(q_i)}} \right)^{c_i}
\end{equation}
where $c_i, i=1, \ldots, \tilde{s}$, and $\left\{ y_{p_i}-y_{q_i} \right\}_{i=1}^{\tilde{s}}$ correspond to the weak kinetic adjustment factors $\tilde{K}_{\rho(i),\rho(i)_K}(\mathbf{x})$ given by (\ref{adjustment}) and Lemma \ref{lemma00}, and the tree constants $\tilde{K}_j$, $j=1, \ldots, \tilde{n}$, are defined according to (\ref{treeconstant}) for $\tilde{G}(\tilde{V},\tilde{E})$. We will say that $\tilde{\mathcal{N}}$ is \textbf{strongly resolvable} if, for every $i \in \mathcal{R}_I$, $\tilde{K}_{\rho(i),\rho(i)_K}$ does not depend on any $\tilde{k}_j$, $j \in \mathcal{R}_I$.
\end{definition}


\noindent Strong resolvability will allow us to finally define a translated mass action system for an improper translation.

\begin{definition}
\label{improperkinetic}
Consider a chemical reaction network $\mathcal{N} = (\mathcal{S},\mathcal{C},\mathcal{R})$ and an associated mass action system $\mathcal{M} = (\mathcal{S},\mathcal{C},\mathcal{R},k)$. Suppose $\tilde{\mathcal{N}} = (\mathcal{S},\tilde{\mathcal{C}},\mathcal{CR}_K,\tilde{\mathcal{R}})$ is a strongly resolvable improper translation of $\mathcal{N}$. Then we define the \textbf{improperly translated mass action system} to be the generalized mass action system $(\mathcal{S},\tilde{\mathcal{C}},\mathcal{C}_K,\tilde{\mathcal{R}},\tilde{k})$ with rate constants
\begin{equation}
\label{111}
\tilde{k}_{h_1(i)} = \left\{ \begin{array}{ll} k_i, \; \; \; \; \; & \mbox{for } i \in \mathcal{R} \setminus \mathcal{R}_I \\ \left( \tilde{K}_{\rho(i),\rho(i)_K} \right) k_i \; \; \; \; \; & \mbox{for } i \in \mathcal{R}_I.\end{array}\right.
\end{equation}
\end{definition}



We are now prepared to make a correspondence between the steady states of a mass action system $\mathcal{M}$ and the generalized mass action system $\tilde{\mathcal{M}}$ (defined by Definition \ref{improperkinetic}) associated with an improper translation $\tilde{\mathcal{N}}$.

\begin{lemma}
 \label{lemma331}
Consider an improper translation $\tilde{\mathcal{N}} = (\mathcal{S},\tilde{\mathcal{C}},\mathcal{CR}_K,\tilde{\mathcal{R}})$ of a chemical reaction network $\mathcal{N} = (\mathcal{S},\mathcal{C},\mathcal{R})$. Suppose that $\tilde{\mathcal{N}}$ is strongly resolvable and $\tilde{\delta}=0$. Let $\mathcal{M} = (\mathcal{S},\mathcal{C},\mathcal{R},k)$ be a mass action system corresponding to $\mathcal{N}$ and $\tilde{\mathcal{M}} = (\mathcal{S},\tilde{\mathcal{C}},\mathcal{CR}_K,\tilde{\mathcal{R}},\tilde{k})$ be an improperly translated mass action system corresponding to $\tilde{\mathcal{N}}$ and defined by Definition \ref{improperkinetic}. Then the steady states of the system (\ref{gde}) governing $\tilde{\mathcal{M}}$ coincide with those of the system (\ref{de}) governing $\mathcal{M}$.
\end{lemma}

\begin{proof}
Consider an improper translation $\tilde{\mathcal{N}} = (\mathcal{S},\tilde{\mathcal{C}},\mathcal{CR}_K,\tilde{\mathcal{R}})$ of a chemical reaction network $\mathcal{N} = (\mathcal{S},\mathcal{C},\mathcal{R})$ which is strongly resolvable. Without loss of generality, we will index the reactions of $\tilde{\mathcal{N}}$ so that $h_1$ may be taken as the identity. Let $\mathcal{M} = (\mathcal{S},\mathcal{C},\mathcal{R},k)$ be a mass action system corresponding to $\mathcal{N}$ and $\tilde{\mathcal{M}} = (\mathcal{S},\tilde{\mathcal{C}},\mathcal{CR}_K,\tilde{\mathcal{R}},\tilde{k})$ be an improperly translated mass action system corresponding to $\tilde{\mathcal{N}}$ and defined by Definition \ref{improperkinetic}. It follows from property $1.$ of Definition \ref{translation} that the system (\ref{de}) governing $\mathcal{M}$ can be written
\begin{equation}
\label{922}
\frac{d\mathbf{x}}{dt} = Y \; I_a \; I_k \; \Psi(\mathbf{x}) = \tilde{Y} \; \tilde{I}_a \; I_k \; \Psi(\mathbf{x})
\end{equation}
where $\tilde{Y}$ and $\tilde{I}_a$ correspond to the translation $\tilde{\mathcal{N}}$.

Now consider the rate vector $R(\mathbf{x}):= I_k \; \Psi(\mathbf{x}) \in \mathbb{R}_{>0}^r$ corresponding to $\mathcal{M}$ and the rate vector $\tilde{R}(\mathbf{x}) := \tilde{I}_k \; \Psi_K(\mathbf{x})$ of $\tilde{\mathcal{M}}$. Since $\tilde{\mathcal{N}}$ is improper, the vector $\Psi_K(\mathbf{x})$ contains a subset of the monomials in $\Psi(\mathbf{x})$ by property $3.$ of Definition \ref{translation}. Consequently, to relate $\mathcal{M}$ and $\tilde{\mathcal{M}}$ we need to remove the monomials in $\Psi(\mathbf{x})$ corresponding to complexes in $\mathcal{CR} \setminus \mathcal{CR}_K$. We will accomplish this by relating the monomials corresponding to complexes in $\mathcal{CR} \setminus \mathcal{CR}_K$ to the monomials in $\mathcal{CR}_K$ and absorbing the corresponding adjustment factor in the matrix $I_k$.

To accomplish this, recall that $\tilde{\mathcal{N}}$ being strongly resolvable implies it is weakly resolvable. Consequently, given the basis $\left\{ y_{p_j}-y_{q_j} \right\}_{j=1}^{\tilde{s}}$ where $p_j, q_j \in \mathcal{CR}$ of $\tilde{S}$, from Lemma \ref{lemma00} it follows that, for every $i \in \mathcal{R}_I$, there are constants $c_j, j =1, \ldots, \tilde{s}$, such that
\begin{equation}
\label{919}
\mathbf{x}^{y_{\rho(i)}} = \left[ \tilde{K}_{\rho(i),\rho(i)_K}(\mathbf{x}) \right] \; \mathbf{x}^{y_{\rho(i)_K}}
\end{equation}
where $\tilde{K}_{\rho(i),\rho(i)_K}(\mathbf{x})$ is given by
\[\tilde{K}_{\rho(i),\rho(i)_K}(\mathbf{x}) = \prod_{j=1}^{\tilde{s}} \left( \frac{\mathbf{x}^{y_{p_j}}}{\mathbf{x}^{y_{q_j}}} \right)^{c_j}.\]

Now define \emph{state dependent} rate functions
 \begin{equation}
\label{999}
\tilde{k}_i(\mathbf{x}) = \left\{ \begin{array}{ll} k_i, \; \; \; \; \; & \mbox{for } i \in \mathcal{R} \setminus \mathcal{R}_I \\ \left( \tilde{K}_{\rho(i),\rho(i)_K}(\mathbf{x}) \right) k_i \; \; \; \; \; & \mbox{for } i \in \mathcal{R}_I. \end{array}\right.
\end{equation}
These rate functions define a state dependent rate constant matrix $\tilde{I}_k(\mathbf{x})$ with entries $[\tilde{I}_k(\mathbf{x})]_{ij} = \tilde{k}_i(\mathbf{x})$ if $h_2(\rho(i))=j$ and $[\tilde{I}_k(\mathbf{x})]_{ij} = 0$ otherwise. We may now use the mass action vector of the translation $\tilde{\mathcal{N}}$, $\Psi_K(\mathbf{x})$, which has entries $[\Psi_K(\mathbf{x})]_{h_2(j)} = \Psi_j(\mathbf{x})$ for $j \in \mathcal{CR}_K$. Define the vector $\tilde{R}_K(\mathbf{x}):= \tilde{I}_k(\mathbf{x}) \; \Psi_K(\mathbf{x})$. It follows by (\ref{919}) and (\ref{999}) that the entries of $\tilde{R}_K(\mathbf{x})$ for $i \in \mathcal{R} \setminus \mathcal{R}_I$ are given by
\[[\tilde{R}_K(\mathbf{x})]_i = k_i \mathbf{x}^{y_{h_2^{-1}(\tilde{\rho}(i))}} = k_i \mathbf{x}^{y_{h_2^{-1}(h_2(\rho(i)))}} = k_i \mathbf{x}^{y_{\rho(i)_K}} = k_i \mathbf{x}^{y_{\rho(i)}} = R_i(\mathbf{x})\]
because $i \in \mathcal{R} \setminus \mathcal{R}_I$ implies $\rho(i)_K= \rho(i)$. For $i \in \mathcal{R}_I$ we have
\[\begin{split} [\tilde{R}_K(\mathbf{x})]_i & = \left( \tilde{K}_{\rho(i),\rho(i)_K}(\mathbf{x}) \right) k_i \mathbf{x}^{y_{h_2^{-1}(\tilde{\rho}(i))}} \\ & = \left( \tilde{K}_{\rho(i),\rho(i)_K}(\mathbf{x}) \right)k_i \mathbf{x}^{y_{\rho(i)_K}} = k_i \mathbf{x}^{y_{\rho(i)}} = R_i(\mathbf{x}) \end{split}\]
by (\ref{919}). Consequently, by (\ref{922}) we have that
\begin{equation}
\label{838}
\frac{d\mathbf{x}}{dt} = \tilde{Y} \; \tilde{I}_a \; I_k \; \Psi(\mathbf{x}) = \tilde{Y} \; \tilde{I}_a \; \tilde{I}_k(\mathbf{x}) \; \Psi_K(\mathbf{x}) = \tilde{Y} \; \tilde{A}_k(\mathbf{x}) \; \Psi_K(\mathbf{x})
\end{equation}
\noindent where $\tilde{A}_k(\mathbf{x}) := \tilde{I}_a \; \tilde{I}_k(\mathbf{x}) \in \mathbb{R}_{>0}^{\tilde{n} \times \tilde{n}}$ is a state dependent kinetic matrix with the same structure as the translation $\tilde{\mathcal{N}}$ and rate functions given by (\ref{999}).


Let $\tilde{G}(\tilde{V},\tilde{E})(\mathbf{x})$ denote the weighted directly graph of $\tilde{\mathcal{N}}$ with state dependent edge weights given by (\ref{999}). In order to remove the state dependence in $\tilde{A}_k(\mathbf{x})$ and $\tilde{G}(\tilde{V},\tilde{E})(\mathbf{x})$, we consider the system at \emph{steady state}. Since $\tilde{\delta} = 0$ for the translation $\tilde{\mathcal{N}}$, it follows that
\begin{equation}
\label{3822}
\tilde{Y} \; \tilde{A}_k(\mathbf{x}) \; \Psi_K(\mathbf{x}) = \mathbf{0} \; \; \; \Longleftrightarrow \; \; \; \tilde{A}_k(\mathbf{x}) \; \Psi_K(\mathbf{x}) = \mathbf{0}.
\end{equation}
Now let $\tilde{\Lambda}_i$, $i=1, \ldots, \tilde{\ell},$ denote the supports of the $\tilde{\ell}$ linkages of $\tilde{\mathcal{N}}$ and $\tilde{K}_j(\mathbf{x})$, $j=1, \ldots, \tilde{n}$, denote the state dependent tree constants (\ref{treeconstant}) of $\tilde{G}(\tilde{V},\tilde{E})(\mathbf{x})$. Since $\tilde{\mathcal{N}}$ is a strong translation, we have that $\tilde{G}(\tilde{V},\tilde{E})(\mathbf{x})$ is weakly reversible. By Theorem \ref{Akernel} we therefore have that
\[\mbox{ker} (\tilde{A}_k(\mathbf{x})) = \mbox{span} \left\{ \tilde{\mathbf{K}}_1(\mathbf{x}), \tilde{\mathbf{K}}_2(\mathbf{x}), \ldots, \tilde{\mathbf{K}}_{\tilde{\ell}}(\mathbf{x}) \right\}\]
where $\tilde{\mathbf{K}}_j(\mathbf{x}) = ([\tilde{K}_j(\mathbf{x})]_1,[\tilde{K}_j(\mathbf{x})]_2,\ldots,[\tilde{K}_j(\mathbf{x})]_{\tilde{n}})$ has entries
\[ [\tilde{K}_j(\mathbf{x})]_i = \left\{ \begin{array}{ll} \tilde{K}_i(\mathbf{x}), \; \; \; \; \; \; & \mbox{if } i \in \Lambda_j \\ 0 & \mbox{otherwise.} \end{array} \right.\]
It follows that, for every $i, j \in \mathcal{CR}_K$ such that $h_2(i), h_2(j) \in \tilde{\mathcal{L}}_k$ for some $k=1, \ldots, \tilde{\ell}$, we have
\begin{equation}
\label{222}
\frac{\mathbf{x}^{y_i}}{\tilde{K}_{h_2(i)}(\mathbf{x})} = \frac{\mathbf{x}^{y_j}}{\tilde{K}_{h_2(j)}(\mathbf{x})} \; \; \; \Longleftrightarrow \; \; \; \frac{\mathbf{x}^{y_i}}{\mathbf{x}^{y_j}} = \frac{\tilde{K}_{h_2(i)}(\mathbf{x})}{\tilde{K}_{h_2(j)}(\mathbf{x})}.
\end{equation}
Since $\tilde{\mathcal{N}}$ is weakly reversible, it follows that $\tilde{S}$ has a basis of the form $\left\{ y_{p_j}-y_{q_j} \right\}_{j=1}^{\tilde{s}}$ where $p_j, q_j \in \mathcal{CR}_K$. Since (\ref{222}) holds for all $i,j \in \mathcal{CR}_K$ such that $h_2(i),h_2(j) \in \tilde{\mathcal{L}}_k$ it holds for every pair $p_j,q_j$ corresponding to a basis element of $\tilde{S}$. That is to say, we have
\[\frac{\mathbf{x}^{y_{p_j}}}{\mathbf{x}^{y_{q_j}}} = \frac{\tilde{K}_{h_2(p_j)}(\mathbf{x})}{\tilde{K}_{h_2(q_j)}(\mathbf{x})}.\]
Now let $c_j$, $j=1, \ldots, \tilde{s}$, denote the coordinates of $y_{\rho(i)}-y_{\rho(i)_K}$ for some $i \in \mathcal{R}_K$ with respect to the basis $\left\{ y_{p_j} - y_{q_j} \right\}_{j=1}^{\tilde{s}}$ of $\tilde{S}$. Then we have
\begin{equation}
\label{333}
\tilde{K}_{\rho(i),\rho(i)_K}(\mathbf{x}) = \left( \prod_{i=1}^{\tilde{s}} \frac{\mathbf{x}^{y_{p_i}}}{\mathbf{x}^{y_{q_i}}} \right)^{c_i} = \left( \prod_{i=1}^{\tilde{s}} \frac{\tilde{K}_{h_2(p_i)}(\mathbf{x})}{\tilde{K}_{h_2(q_i)}(\mathbf{x})} \right)^{c_i}.
\end{equation}

Clearly, the product on the far right of (\ref{333}) is state dependent. Since $\tilde{\mathcal{N}}$ is strongly resolvable, however, we know that, for every $i \in \mathcal{R}_I$,
\begin{equation}
\label{33334}
\tilde{K}_{\rho(i),\rho(i)_K} = \prod_{j=1}^{\tilde{s}} \left( \frac{\tilde{K}_{h_2(p_j)}}{\tilde{K}_{h_2(q_j)}} \right)^{c_j}
\end{equation}
does not depend on any $\tilde{k}_j$ corresponding to $j \in \mathcal{R}_I$ where the tree constants $\tilde{K}_j$ are determined with respect to the semi-proper reaction graph $\tilde{G}(\tilde{V},\tilde{E})$ of $\tilde{\mathcal{N}}$. Since $\tilde{G}(\tilde{V},\tilde{E})$ and $\tilde{G}(\tilde{V},\tilde{E})(\mathbf{x})$ both correspond structurally to $\tilde{\mathcal{N}}$, the tree constants have the same dependency on edge weights. It follows that
\begin{equation}
\label{33333}
\left( \prod_{j=1}^{\tilde{s}} \frac{\tilde{K}_{h_2(p_j)}(\mathbf{x})}{\tilde{K}_{h_2(q_j)}(\mathbf{x})} \right)^{c_j}
\end{equation}
does not depend on any on any $\tilde{k}_j(\mathbf{x})$ corresponding to $j \in \mathcal{R}_I$. However, the only state dependent rate functions in (\ref{999}) corresponded to $j \in \mathcal{R}_I$. It follows that (\ref{33333}) does not depend on any state dependent rate constant in (\ref{999}). It follows that (\ref{333}), and consequently (\ref{999}), are independent of the state $\mathbf{x} \in \mathbb{R}_{>0}^m$.

Notice furthermore that (\ref{33334}) and (\ref{33333}) only depend on $k_i$ corresponding to $i \in \mathcal{R} \setminus \mathcal{R}_I$. Since the edge weights of $\tilde{G}(\tilde{V},\tilde{E})(\mathbf{x})$ and $\tilde{G}(\tilde{V},\tilde{E})$ coincide for $k_i$ corresponding to $i \in \mathcal{R} \setminus \mathcal{R}_I$, it follows from (\ref{333}), (\ref{33334}) and (\ref{33333}) that we have $\tilde{K}_{\rho(i),\rho(i)_K}(\mathbf{x}) = \tilde{K}_{\rho(i),\rho(i)_K}$. It then follows from (\ref{999}) that, at steady state, we have
\[\tilde{k}_i(\mathbf{x}) = \left\{ \begin{array}{ll} k_i, \; \; \; \; \; & \mbox{for } i \in \mathcal{R} \setminus \mathcal{R}_I \\ \left( \tilde{K}_{\rho(i),\rho(i)_K} \right) k_i \; \; \; \; \; & \mbox{for } i \in \mathcal{R}_I. \end{array}\right.\]
This corresponds to the choice of rate constants for the improperly translated mass action system $\tilde{\mathcal{M}} = (\mathcal{S},\tilde{\mathcal{C}},\mathcal{CR}_K,\tilde{\mathcal{R}},\tilde{k})$ defined by Definition \ref{improperkinetic}. Consequently, from (\ref{838}) we have
\begin{equation}
\label{1000}
\tilde{Y} \; \tilde{I}_a \; \tilde{I}_k(\mathbf{x}) \; \Psi_K(\mathbf{x}) = \tilde{Y} \; \tilde{I}_a \; \tilde{I}_k \; \Psi_K(\mathbf{x})
\end{equation}
so that the system (\ref{de}) governing $\mathcal{M}$ and the system (\ref{gde}) governing $\tilde{\mathcal{M}}$ defined by Definition \ref{improperkinetic} coincide at steady state, and we are done.
\end{proof}

\subsection{Connection with Toric Steady States}
\label{section4}

In order to relate translated mass action systems, complex balanced generalized mass action systems, and toric steady states, we need to first understand how the kinetic spaces associated with $\mathcal{N}$ and its translation $\tilde{\mathcal{N}}$ are related.

\begin{lemma}
\label{lemma232}
Consider a chemical reaction network $\mathcal{N} = (\mathcal{S},\mathcal{C},\mathcal{R})$ and a strong translation $\tilde{\mathcal{N}} = (\mathcal{S},\tilde{\mathcal{C}},\mathcal{CR}_K,\tilde{\mathcal{R}})$. Then the stoichiometric subspaces $S$ of $\mathcal{N}$ and $\tilde{\mathcal{N}}$ coincide and the kinetic-order subspace $\tilde{S}$ of $\tilde{\mathcal{N}}$ is given by
\begin{equation}
\label{tildeS}
\tilde{S} = \mbox{span} \left\{ \bigcup_{k=1}^{\tilde{\ell}} \left\{ y_{p} - y_{q} \; | \; p,q \in \mathcal{CR}_K, \; h_2(p), h_2(q) \in \tilde{\mathcal{L}}_k \right\} \right\}.
\end{equation}
\end{lemma}


\begin{proof}
Let $\mathcal{N} = (\mathcal{S},\mathcal{C},\mathcal{R})$ be a chemical reaction network and  $\tilde{\mathcal{N}} = (\mathcal{S},\tilde{\mathcal{C}},\mathcal{CR}_K,\tilde{\mathcal{R}})$ be a strong translation of $\mathcal{N}$. By property $2.$ of Definition \ref{translation}, $\mathcal{N}$ and $\tilde{\mathcal{N}}$ have the same reaction vectors and therefore have the same stoichiometric subspace $S$. This proves the first claim.

To the second claim, we recall that a strong translation $\tilde{\mathcal{N}}$ is weakly reversible and therefore only contains reactant complexes. These reactant complexes are assigned kinetic complexes from the set $\mathcal{CR}$ according to the relation $h_2$. It is well known that the span of the reaction vectors of a chemical reaction network is the same as the span of the stoichiometric differences of complexes on the same connected component (for example, see pg. 189 of \cite{F1}). Since the kinetic complexes are drawn bijectively from $\mathcal{CR}_K$ by $h_2$, this completes the proof.
\end{proof}

\begin{remark}
This result says that the stoichiometric subspaces of a network and its translation are the same and the kinetic-order subspace of the translation is given by the span of the stoichiometric differences of the kinetically relevant complexes which map through $h_2$ to the same connected components of $\tilde{\mathcal{N}}$.
\end{remark}


We can now make the following connection between translated mass action systems, complex balanced generalized mass action systems, and systems with toric steady states. The following is the main result of the paper.

\begin{theorem}
\label{theorem01}
Suppose $\tilde{\mathcal{N}} = (\mathcal{S},\tilde{\mathcal{C}},\mathcal{CR}_K,\tilde{\mathcal{R}})$ is a translation of a chemical reaction network $\mathcal{N} = (\mathcal{S},\mathcal{C},\mathcal{R})$ which is either strong and proper, or strongly resolvably improper. Let $\tilde{\mathcal{M}} = (\mathcal{S},\tilde{\mathcal{C}},\mathcal{CR}_K,\tilde{\mathcal{R}},\tilde{k})$ denote a properly or improperly translated mass action system corresponding to $\mathcal{M} = (\mathcal{S},\mathcal{C},\mathcal{R},k)$ (Definition \ref{properkinetic} or Definition \ref{improperkinetic}). Suppose furthermore that $\tilde{\mathcal{N}}$ satisfies $\tilde{\delta} = \tilde{\delta}_K = 0$. Then:
\begin{enumerate}
\item
The mass action system $\mathcal{M}$ has toric steady states for all rate constant vectors $k \in \mathbb{R}_{>0}^r$.
\item
The toric steady state ideal of $\mathcal{M}$ is generated by the binomials
\[\tilde{K}_{h_2(i)} \mathbf{x}^{y_{j}} - \tilde{K}_{h_2(j)} \mathbf{x}^{y_{i}}\]
for $i,j \in \mathcal{CR}_K$ such that $h_2(i), h_2(j) \in \tilde{\mathcal{L}}_k$, $k=1, \ldots, \tilde{\ell},$ and $\tilde{K}_{h_2(i)}$, $i =1, \ldots, n$, are the tree constants (\ref{treeconstant}) corresponding to the translated reaction graph of $\tilde{\mathcal{N}}$.
\item
The toric steady states of $\mathcal{M}$ correspond to the complex balanced steady states of $\tilde{\mathcal{M}}$ and can be parametrized by
\[E = \left\{ \mathbf{x} \in \mathbb{R}_{>0}^m \; | \; \ln(\mathbf{x}) - \ln(\mathbf{x}^*) \in \tilde{S}^{\perp} \right\}\]
where
\[\tilde{S} = \mbox{span} \left\{ \bigcup_{k=1}^{\tilde{\ell}} \left\{ y_{p} - y_{q} \; | \; p,q \in \mathcal{CR}, \; h_2(p), h_2(q) \in \tilde{\mathcal{L}}_k\right\} \right\}.\]

\item
If $\sigma(S) = \sigma(\tilde{S})$ and $(+,\cdots,+) \in \sigma(S^{\perp})$ then there is exactly one toric steady state within each stoichiometric compatibility class $\mathsf{C}_{\mathbf{x}_0} = (\mathbf{x}_0 + S) \cap \mathbb{R}^m$ of $\mathcal{M}$.
\item
If $\sigma(S) \cap \sigma(\tilde{S}^{\perp}) \not= \left\{ 0 \right\}$ then there exists a rate constant vector $k \in \mathbb{R}_{>0}^r$ such that $\mathcal{M}$ has more than one toric steady state in some stoichiometric compatibility class of $\mathcal{M}$.
\end{enumerate}
\end{theorem}

\begin{proof}[Proof (1-3):]
Let $\mathcal{N} = (\mathcal{S},\mathcal{C},\mathcal{R})$ be a chemical reaction network and $\tilde{\mathcal{N}} = (\mathcal{S},\tilde{\mathcal{C}},\mathcal{CR},\tilde{\mathcal{R}})$ be a translation of $\mathcal{N}$ which is either strong and proper, or strongly resolvable improper. Suppose $\mathcal{M} = (\mathcal{S},\mathcal{C},\mathcal{R},k)$ is a mass action system corresponding to $\mathcal{N}$. We define the translated mass action system $\mathcal{\tilde{M}} = (\mathcal{S},\tilde{\mathcal{C}},\mathcal{CR}_K,\tilde{\mathcal{R}},\tilde{k})$ according to Definition \ref{properkinetic} if $\tilde{\mathcal{N}}$ is a proper translation of $\mathcal{N}$, and by Definition \ref{improperkinetic} if $\tilde{\mathcal{N}}$ is a strongly resolvable improper translation of $\mathcal{N}$.

From either Lemma \ref{lemma231} and Lemma \ref{lemma331} we have that the steady state set of $\tilde{\mathcal{M}}$ corresponds to the steady state set of $\mathcal{M}$. Correspondingly, by either (\ref{382}) or (\ref{1000}), if we take $h_1$ to be the identity, we have that
\[\frac{d\mathbf{x}}{dt} = \tilde{Y} \; \tilde{I}_a \; \tilde{I}_k \; \Psi_K(\mathbf{x}) = \tilde{Y} \; \tilde{A}_k \; \Psi_K(\mathbf{x})\]
where the rate constants of $\tilde{I}_k$ and $\tilde{A}_k:= \tilde{I}_a \; \tilde{I}_k$ are defined by $\tilde{k}_{h_1(i)} = k_i$ is $\tilde{\mathcal{N}}$ is proper and (\ref{111}) if $\tilde{\mathcal{N}}$ is strongly resolvably improper, and $\Psi_K(\mathbf{x})$ has entries $[\Psi_K(\mathbf{x})]_{h_2(j)} = \Psi_j(\mathbf{x})$ for $j \in \mathcal{CR}_K$. (Notice that, for proper translations, $\mathcal{CR}_K = \mathcal{CR}$ and $h_2$ is bijective so that this coincides with the definition of $\Psi_K(\mathbf{x})$ given in the proof of Lemma \ref{lemma231}.)

Since $\tilde{\delta}_K = 0$, we may conclude by Proposition 2.20 of \cite{M-R} that the translated mass action system $\tilde{\mathcal{M}}$ has a complex balanced steady state. That is to say, there is a point $\mathbf{x}^*$ which satisfies
\begin{equation}
\label{kernel2}
\Psi_{K}(\mathbf{x}^*) \in \mbox{ ker}(\tilde{A}_k).
\end{equation}
Furthermore, since $\tilde{\delta} = $ dim(ker$(\tilde{Y}) \cap$ Im$(\tilde{I}_a))=0$, we have from (\ref{3822}) that all steady states are complex balanced steady states. It follows from Proposition 2.21 of \cite{M-R} the set of such steady states may be parametrized by 
\[E = \left\{ \mathbf{x} \in \mathbb{R}_{>0}^m \; | \; \ln(\mathbf{x}) - \ln(\mathbf{x}^*) \in \tilde{S} \right\}\]
where
\[\tilde{S} = \mbox{span} \left\{ \bigcup_{k=1}^{\tilde{\ell}} \left\{ y_{p} - y_{q} \; | \; p,q \in \mathcal{CR}, \; h_2(p), h_2(q) \in \tilde{\mathcal{L}}_k\right\} \right\}.\]
by Lemma \ref{lemma232}. This is sufficient to prove claim $3.$

Now consider claims $1.$ and $2$. Since $\tilde{\mathcal{N}}$ is a strong translation it is weakly reversible. Consequently, by Theorem \ref{Akernel}, we may conclude that
\begin{equation}
\label{543}
\mbox{ker} (\tilde{A}_k) = \mbox{span} \left\{ \tilde{\mathbf{K}}_1, \tilde{\mathbf{K}}_2, \ldots, \tilde{\mathbf{K}}_{\tilde{\ell}} \right\}
\end{equation}
where $\tilde{\mathbf{K}}_j = ([\tilde{K}_j]_1,[\tilde{K}_j]_2,\ldots,[\tilde{K}_j]_{\tilde{n}})$ has entries
\begin{equation}
\label{544}
 [\tilde{K}_j]_i = \left\{ \begin{array}{ll} \tilde{K}_i, \; \; \; \; \; \; & \mbox{if } i \in \Lambda_j \\ 0 & \mbox{otherwise} \end{array} \right.
\end{equation}
where $\Lambda_j$ denotes the support of the $j^{th}$ linkage class $\tilde{\mathcal{L}}_j$ in $\tilde{\mathcal{N}}$, and $\tilde{K}_i$ is the tree constant of the $i^{th}$ complex $\tilde{\mathcal{C}}_i$ in the translated reaction graph $\tilde{G}(\tilde{V},\tilde{E})$.

It follows from (\ref{kernel2}), (\ref{543}), and (\ref{544}) that, for every $i, j \in \mathcal{CR}_K$ such that $h_2(i), h_2(j) \in \tilde{\mathcal{L}}_k$ for some $k=1, \ldots, \tilde{\ell}$, the steady states $\mathbf{x}^* \in \mathbb{R}_{>0}^m$ satisfy
\[\frac{(\mathbf{x}^*)^{y_i}}{\tilde{K}_{h_2(i)}} = \frac{(\mathbf{x}^*)^{y_j}}{\tilde{K}_{h_2(j)}} \; \; \; \Longleftrightarrow \; \; \; \tilde{K}_{h_2(j)} (\mathbf{x}^*)^{y_i} - \tilde{K}_{h_2(i)} (\mathbf{x}^*)^{y_j} = 0.\]
Since this set corresponds to the steady states of $\mathcal{M}$ by either Lemma \ref{lemma231} or Lemma \ref{lemma331}, we have shown that $\mathcal{M}$ has toric steady states generated by binomials of the form required of claim $2$. Since the choice of rate constants in the definition of $\mathcal{M}$ was arbitrary, claims $1.$ and $2.$ follow.\\ \\
\emph{Proof \;(4-5):} This follows directly from claims $1.$ through $3.$ of Theorem \ref{theorem01}, Lemma \ref{lemma232}, and Proposition 3.2 and Theorem 3.10 of \cite{M-R}.
\end{proof}

\section{Techniques and Examples}
\label{examplesection}


In general, when applying Theorem \ref{theorem01} we do not have a translation $\tilde{\mathcal{N}}$ of $\mathcal{N}$ given to us; rather, we must \emph{find} it. In this section, a simple heuristic method for generating translated chemical reaction networks will be presented. We follow this discussion with three examples which are known to contain toric steady states \cite{M-D-S-C}.


\subsection{Translation Algorithm}
\label{techniquesection}



We make several observations about the process of network translations, in particular as it relates the assumptions necessary to apply Theorem \ref{theorem01}. We firstly require that the translation $\tilde{\mathcal{N}}$ is strong and that the structural deficiency is zero (i.e. $\tilde{\delta} =0$). It follows from the discussion in Section \ref{generatorsection} that the translation $\tilde{\mathcal{N}}$ must not have any stoichiometric generators of ker$(\Gamma) \cap \mathbb{R}_{\geq 0}^r$.

We notice, however, that Definition \ref{translation} implies any translation $\tilde{\mathcal{N}}$ shares the same reaction vectors as $\mathcal{N}$ and, consequently, $\Gamma := Y \; I_a = \tilde{Y} \; \tilde{I}_a$. It follows that the original network $\mathcal{N}$ and the translation $\tilde{\mathcal{N}}$ must share the same generators $\left\{ E_1, \ldots, E_f \right\}$. When constructing translations $\tilde{\mathcal{N}}$ for the purpose of applying Theorem \ref{theorem01}, we therefore have the following key intuition:
\begin{center}
\parbox[c]{4.5in}{The process of network translation must turn the stoichiometric generators of ker$(Y \; I_a) \cap \mathbb{R}_{\geq 0}^r$ into cyclic generators of ker$(\tilde{Y} \; \tilde{I}_a) \cap \mathbb{R}_{\geq 0}^r$.}
\end{center}
We propose the following technique for constructing translated chemical reaction networks $\tilde{\mathcal{N}}$ which can be used to characterize the steady states of mass action systems $\mathcal{M}$ through Theorem \ref{theorem01}.\\

\underline{Translation Algorithm}:\\
\begin{enumerate}
\item
Identify the stoichiometric generators of ker$(Y \; I_a) \cap \mathbb{R}_{\geq 0}^r$.
\item
For each stoichiometric generator $E_i = (E_{i1}, E_{i2}, \ldots, E_{ir})$ identified in part $1.$ do the following:
\begin{enumerate}
\item
Assign the reactions on the support of the generator an ordering $\left\{ \mu(1), \right.$ $\ldots,$ $\left. \mu(q) \right\} \subseteq \left\{ 1, \ldots, r \right\}$.
\item
If possible, add and/or subtract species to the left- and right-hand side of each $\mathcal{R}_{\mu(i)}$, $i = 1, \ldots, q$, so that that $\mathcal{C}_{\rho(\mu(i))} \to \mathcal{C}_{\rho'(\mu(i))}$ becomes $\tilde{\mathcal{C}}_{\rho(\mu(i))} \to \tilde{\mathcal{C}}_{\rho'(\mu(i))}$ where the new complexes satisfy $\tilde{\mathcal{C}}_{\rho'(\mu(i-1))} = \tilde{\mathcal{C}}_{\rho(\mu(i))}$ for all $i=1, \ldots, q$ (allowing $\mu(q) = \mu(0))$. This forms a cycle
\[\tilde{\mathcal{C}}_{\rho(\mu(1))} \to \tilde{\mathcal{C}}_{\rho(\mu(2))} \to \cdots \to \tilde{\mathcal{C}}_{\rho(\mu(q))} \to \tilde{\mathcal{C}}_{\rho(\mu(1))}.\]
\end{enumerate}
\item
Translate the following reactions in unison (i.e. add/subtract the same factors in step $2(b)$):
\begin{enumerate}
\item
Any reactions with the same source complex (i.e. any reactions $\mathcal{R}_i$, $\mathcal{R}_j$ for which $\rho(i) = \rho(j)$).
\item
Any reactions already on the support of a cyclic generator in $\mathcal{N}$.
\end{enumerate}
\item
For each reaction $\mathcal{R}_i$, assign the original source complex $\mathcal{C}_{\rho(i)}$ as the kinetic complexes of the new complexes $\tilde{\mathcal{C}}_{\rho(i)}$. In the case of multiple source complexes being assigned to a new complex, any original source complex may be chosen.
\end{enumerate}
If successful, this algorithm produces a strongly translated chemical reaction network by Definition \ref{translation} and Definition \ref{supp} with $h_1$ the identity mapping.

The algorithm is deficient in several ways. Most glaringly, there is no guarantee it will work. The stoichiometric and cyclic generators may overlap in such a way that a reconstruction of the form demanded by step $2.$ is not possible. Worse still, even if the algorithm can work it may not work for all choices of reaction orderings in step $2(a)$. Certain orderings may allow the multiple stoichiometric generators to be fitted together while certain others may not. It is a significant combinatorial problem to ask which of the $(q-1)!$ orderings are worth considering and which are not.

For the purposes of this paper, however, we will ignore these subtleties and consider how the translation algorithm can be used intuitively to construct translations through a series of biochemical examples. Consideration of the full potential and limitations of the translation algorithm provided here will be left as the subject for future work.

\subsection{Example I: Futile Cycle}
\label{futilecyclesection}

Consider the enzymatic network $\mathcal{N}$ given by
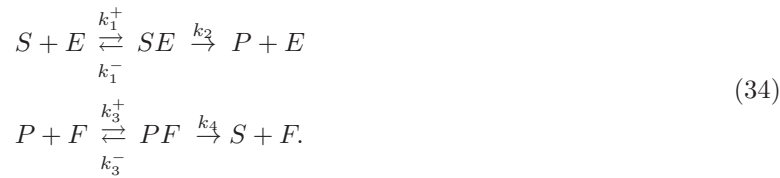
\begin{equation}
\label{system1}
\begin{split} & S + E \; \mathop{\stackrel{k_1^+}{\rightleftarrows}}_{k_1^-} \; SE \; \stackrel{k_2}{\rightarrow} \; P + E \\ & P + F \; \mathop{\stackrel{k_3^+}{\rightleftarrows}}_{k_3^-} \; PF \; \stackrel{k_4}{\rightarrow} S + F.\end{split}
\end{equation}
\noindent This network describes an enzymatic mechanism where one enzyme $E$ catalyzes the transition of a substrate $S$ into a product $P$, and a separate enzyme $F$ catalyzes the reverse transition. Due to the appearance that the two enzymes are competing to undo the work of the other, this network is often called the \emph{futile cycle} \cite{A-S2,W-S,M-D-S-C}. 

The steady states of this network under mass action (and more general) kinetics has been well-studied in the mathematical literature. The most thoroughly discussion is given in \cite{A-S2}, where the authors show through a monotonicity argument that, for every choice of rate constants, every stoichiometric compatibility class $\mathsf{C}_{\mathbf{x}_0}$ of (\ref{system1}) contains precisely one positive steady state and that this steady state is globally asymptotically stable relative to $\mathsf{C}_{\mathbf{x}_0}$. It has also be shown by the deficiency one algorithm \cite{Fe4}, the main argument of \cite{W-S}, and Theorem 5.5 of \cite{M-D-S-C} that the network may not permit multistationarity.

A notable absence in the list of applicable theories is the Deficiency Zero Theorem (Theorem \ref{deficiencyzerotheorem}). The network (\ref{system1}) seems like it should be a prime example of complex balancing, since there are very clear flux modes (i.e. sequences of reactions) which are balanced at steady state. Nevertheless, the network is neither weakly reversible nor deficiency zero, and therefore these balanced steady states may not be related to complex balanced steady states by Theorem \ref{deficiencyzerotheorem}. We will now show that deficiency theory \emph{does} apply but not to the original network $\mathcal{N}$; rather, it applies to a specially constructed \emph{translation} $\tilde{\mathcal{N}}$ of $\mathcal{N}$.

For indexing purposes, we will let the species set $\mathcal{S}$ be given by $\mathcal{A}_1 = S$, $\mathcal{A}_2 = E$, $\mathcal{A}_3 = SE$, $\mathcal{A}_4 = P$, $\mathcal{A}_5 = F$, and $\mathcal{A}_6 = PF$, and the complex set $\mathcal{C}$ be indexed by $\mathcal{C}_1 = \mathcal{A}_1 + \mathcal{A}_2$, $\mathcal{C}_2 = \mathcal{A}_3$, $\mathcal{C}_3 = \mathcal{A}_2 + \mathcal{A}_4$, $\mathcal{C}_4 = \mathcal{A}_4 + \mathcal{A}_5$, $\mathcal{C}_5 = \mathcal{A}_6$, and $\mathcal{C}_6 = \mathcal{A}_1 + \mathcal{A}_5$. We will let the reactions be ordered according to the ordering of the rate constants $\left\{ k_1^+,k_1^-,k_2,k_3^+,k_3^-,k_4 \right\}$. We notice that the reactant complex set is $\mathcal{CR} = \left\{ 1, 2, 4, 5 \right\} \subset \mathcal{C}$ and the reaction profile is $(\sigma(1),\sigma(2),\sigma(3),\sigma(4),\sigma(5),\sigma(6))=(1,2,2,4,5,5)$.

We now seek to apply the translation algorithm to $\mathcal{N}$. It can be easily computed that there are three generators of the current cone ker$(\Gamma) \cap \mathbb{R}_{\geq 0}^r$: the two cyclic generators $E_1 = [1 \; 1 \; 0 \; 0 \; 0 \; 0]^T$ and $E_2 = [0 \; 0 \; 0 \; 1 \; 1 \; 0 ]^T$, and the stoichiometric generator $E_3 = [1 \; 0 \; 1 \; 1 \; 0 \; 1]^T$.

We assign the ordering $\left\{1, 3, 4, 6 \right\}$ to the reactions on the support of $E_3$ by part $2(a)$ of the algorithm. By part $2(b)$, we must translate the reactions so that $\tilde{\mathcal{C}}_{\rho'(3)} = \tilde{\mathcal{C}}_{\rho(4)}$ and $\tilde{\mathcal{C}}_{\rho'(6)} = \tilde{\mathcal{C}}_{\rho(1)}$. One admissible choice is
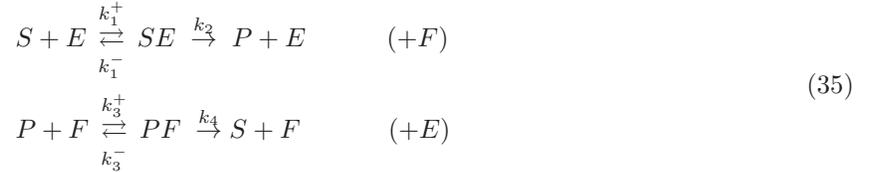
\begin{equation}
\label{futilecycletranslation}
\begin{split} & S + E \; \mathop{\stackrel{k_1^+}{\rightleftarrows}}_{k_1^-} \; SE \; \stackrel{k_2}{\rightarrow} \; P + E \; \; \; \; \; \; \; \; \; \; \; (+F)\\ & P + F \; \mathop{\stackrel{k_3^+}{\rightleftarrows}}_{k_3^-} \; PF \; \stackrel{k_4}{\rightarrow} S + F \; \; \; \; \; \; \; \; \; \; \; \; (+E)\end{split}
\end{equation}
where the indicated additions apply to the entire linkage classes. This choice satisfies part $3.$ of the algorithm. This yields the translation $\tilde{\mathcal{N}} = (\mathcal{S},\tilde{\mathcal{C}},\mathcal{CR},\tilde{\mathcal{R}})$ given by
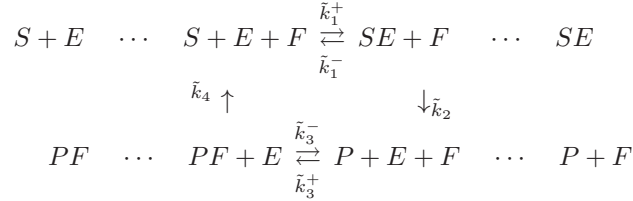
\[\begin{array}{c} \displaystyle{S + E \; \; \; \; \cdots \; \; \; \; S + E + F \; \mathop{\stackrel{\tilde{k}_1^+}{\rightleftarrows}}_{\tilde{k}_1^-} \; SE + F \; \; \; \; \; \cdots \; \; \; \; SE \; \; \; \; \;} \\ {}^{\tilde{k}_4} \uparrow \; \; \; \; \; \; \; \; \; \; \; \; \; \; \; \; \; \; \; \; \; \; \; \; \; \downarrow_{\tilde{k}_2} \\ \displaystyle{\; \; \; \; \; PF \; \; \; \; \cdots \; \; \; \; PF + E \; \mathop{\stackrel{\tilde{k}_3^-}{\rightleftarrows}}_{\tilde{k}_3^+} \; P + E + F \; \; \; \; \cdots \; \; \; \; P + F} \end{array}\]
where the kinetic complexes associated to the complexes in $\tilde{\mathcal{C}}$ are given by the source complexes of the pre-image of the translation. Notice that the generator $E_3$ corresponds to the support of a cycle in $\tilde{\mathcal{N}}$.

Since each source complex in $\mathcal{N}$ is mapped to a unique source complex in $\tilde{\mathcal{N}}$, the translation is proper and therefore, by Definition \ref{properkinetic}, we may define the rate constants of the translated mass action system $\tilde{\mathcal{M}} = (\mathcal{S},\tilde{\mathcal{C}},\mathcal{CR},\tilde{\mathcal{R}},\tilde{k})$ directly with those of $\mathcal{M} = (\mathcal{S},\mathcal{C},\mathcal{R},k)$. That is to say, we can take $\tilde{k}_i^{+ / -} = k_i^{+ / -}$ for $i=1, 2, 3, 4$. Since $\tilde{\mathcal{N}}$ satisfies $\tilde{\delta}=\tilde{\delta}_K = 0$, by claim $1.$ of Theorem \ref{theorem01} we have that $\mathcal{M}$ has toric steady states for all values of the rate constants.

Furthermore, we can characterize these toric steady states by noting that the translated Kirchhoff matrix $\tilde{A}_k$ is
\begin{equation}
\label{33}
\tilde{A}_k = \left[ \begin{array}{cccc} -k_1^+ & k_1^- & 0 & k_4 \\ k_1^+ & - k_1^- - k_2 & 0 & 0 \\ 0 & k_2 & -k_3^+ & k_3^- \\ 0 & 0 & k_3^+ & -k_3^- - k_4 \end{array} \right].
\end{equation}
We can easily compute ker$(\tilde{A}_k)$ according to Theorem \ref{Akernel} to get
\[\begin{split} \tilde{K}_1 & = (k_1^- + k_2)k_3^+k_4 \\ \tilde{K}_2 & = k_1^+k_3^+k_4 \\ \tilde{K}_3 & = k_1^+k_2(k_3^-+k_4) \\ \tilde{K}_4 & = k_1^+k_2k_3^+. \end{split}\]
It follows by claim $2.$ of Theorem \ref{theorem01} that the steady state ideal of $\mathcal{M}$ is generated by the binomials
\[\tilde{K}_2 x_1 x_2 - \tilde{K}_1 x_3, \tilde{K}_3 x_1 x_2 - \tilde{K}_1 x_4 x_5, \mbox{ and } \tilde{K}_4 x_1 x_2 - \tilde{K}_1 x_6.\]
By claim $3.$ of Theorem \ref{theorem01}, this set can be parametrized by rearranging
\[E = \left\{ \mathbf{x} \in \mathbb{R}^6_{> 0} \; | \; \ln(\mathbf{x})-\ln(\mathbf{x}^*) \in \tilde{S}^{\perp} \right\}\]
where
\[\tilde{S} = \mbox{span}\left\{ y_2 - y_1, y_4 - y_1, y_5 - y_1 \right\}\]
for $\mathcal{C}_1, \mathcal{C}_2, \mathcal{C}_4, \mathcal{C}_5 \in \mathcal{CR}$.

It can be easily checked that $\sigma(S) \cap \sigma(\tilde{S}^{\perp}) = \left\{ \mathbf{0} \right\}$ but that $\sigma(S)$ and $\sigma(\tilde{S})$ are not sign-compatible (since no vector with the sign pattern $(0, +, -, +, 0, 0)$ corresponding to the reaction vector of $SE \to P + E$ can be produced by a linear combination of vectors in $\tilde{S}$). Consequently, by claims $4.$ and $5.$ of Theorem \ref{theorem01} we have that $\mathcal{M}$ may not permit multistationarity but the theory falls silent on the whether each stoichiometric compatibility class permits a steady state. For this technicality, we must defer to the results of \cite{A-S2}.

The key observation is that this result allows an explicit characterization of the steady set (\ref{equilibrium}) of $\mathcal{M}$ based on knowledge of the topological network structure of the translation $\tilde{\mathcal{N}}$. This clarifies the connection between established deficiency-based approaches and the newer algebraic work contained in \cite{M-D-S-C}. The trick is to not apply deficiency theory to the original network $\mathcal{N}$; rather, we apply it to a translation $\tilde{\mathcal{N}}$.

\subsection{Example II: Multiple futile cycle}

Now consider the multiple futile cycle $\mathcal{N}$ given by
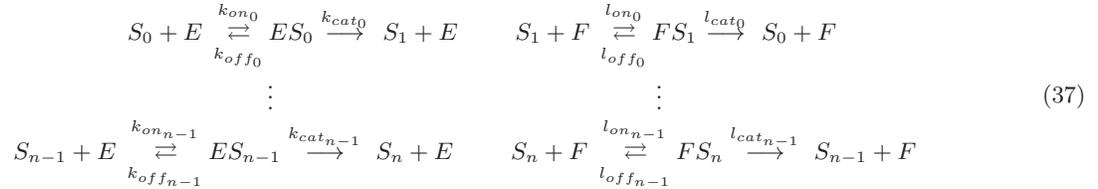
\footnotesize
\begin{equation}
\label{system10}
\begin{array}{c} \displaystyle{\; \; \; \; \; S_0 + E \; \mathop{\stackrel{k_{on_0}}{\rightleftarrows}}_{k_{off_0}} ES_0 \stackrel{k_{cat_0}}{\longrightarrow} \; S_1 + E \hspace{0.3in} S_1 + F \; \mathop{\stackrel{l_{on_0}}{\rightleftarrows}}_{l_{off_0}} FS_1 \stackrel{l_{cat_0}}{\longrightarrow} \; S_0 + F} \\
\displaystyle{\vdots \hspace{2in} \vdots}\\
\displaystyle{S_{n-1} + E \; \mathop{\stackrel{k_{on_{n-1}}}{\rightleftarrows}}_{k_{off_{n-1}}} ES_{n-1} \stackrel{k_{cat_{n-1}}}{\longrightarrow} \; S_n + E \hspace{0.3in} S_n + F \; \mathop{\stackrel{l_{on_{n-1}}}{\rightleftarrows}}_{l_{off_{n-1}}} FS_n \stackrel{l_{cat_{n-1}}}{\longrightarrow} \; S_{n-1} + F} \end{array}
\end{equation}
\small
\noindent This network is a generalization of the futile cycle analyzed in Section \ref{futilecyclesection}. In this network, one enzyme $E$ facilitates a forward cascade of transitions from substrate $S_0$ to substrate $S_n$ while another enzyme $F$ facilitates the reverse transitions. This network has been frequently used in the literature to model multisite sequentially distributed phosphorylation networks of unspecified length \cite{Gun2,M-G,H-F-C}.

Despite the structural similarities with (\ref{system1}), there are notable dynamical differences in the corresponding mass action systems $\mathcal{M}$. It was first shown in \cite{M-H-K} that, even for the simple case $n=2$, the system (\ref{system10}) admits rate constant vectors $k \in \mathbb{R}_{>0}^r$ for which $\mathcal{M}$ exhibits multistationarity. A subsequent focused study of the case $n=2$ in \cite{C-F-R} provided a detailed characterization of the rate parameters which permit multistationarity.

The network (\ref{system10}) has also been studied for arbitrary values of $n \geq 1$ in \cite{Gun2,Gun,W-S,M-D-S-C}. It is known that, for all $n \geq 2$, the associated mass action systems $\mathcal{M}$ admit rate constant vectors $k \in \mathbb{R}_{>0}^r$ for which multistationarity is permitted and that an upper bound on the number of steady states in each compatibility class is $2n-1$ \cite{W-S}. It is furthermore conjectured that this upper bound can be tightened to $n+1$ for odd $n$, and $n$ for even $n$. In \cite{M-D-S-C} the authors prove that, for all rate constant vectors $k \in \mathbb{R}_{>0}^r$, the mass action system $\mathcal{M}$ associated with this network has toric steady states. The authors then explicitly calculate the steady state ideal in terms of those rate constants.

We now show that the steady state set derived in \cite{M-D-S-C} is identical to the steady state set of a particular translation $\tilde{\mathcal{N}}$ of $\mathcal{N}$. To apply step $1.$ of the translation algorithm in Section \ref{techniquesection}, we identify the stoichiometric generators of ker$(Y \; I_a) \cap \mathbb{R}_{\geq 0}^r$. We first divide the reaction network into $n$ subcomponents $\mathcal{N}_i$ of the form
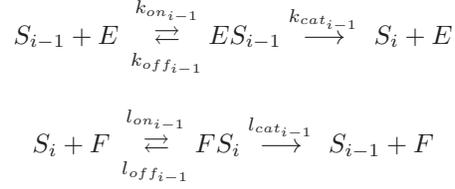
\[S_{i-1} + E \; \mathop{\stackrel{k_{on_{i-1}}}{\rightleftarrows}}_{k_{off_{i-1}}} ES_{i-1} \stackrel{k_{cat_{i-1}}}{\longrightarrow} \; S_i + E\]
\[S_i + F \; \mathop{\stackrel{l_{on_{i-1}}}{\rightleftarrows}}_{l_{off_{i-1}}} FS_i \stackrel{l_{cat_{i-1}}}{\longrightarrow} \; S_{i-1} + F\]
for $i=1, \ldots, n$. This subnetwork is identical to the futile cycle network (\ref{system1}) and, consequently, on the support of the reactions in $\mathcal{N}_i$ we have the single stoichiometric generator $E_i = (\cdots \; | \; 1, 0, 1, 1, 0, 1 \; | \; \cdots)$.

We perform step $2.$ of the translation algorithm in the same manner as in Section \ref{futilecyclesection}, with slightly different terms. We translate each subnetwork $\mathcal{N}_i$ in the following way:
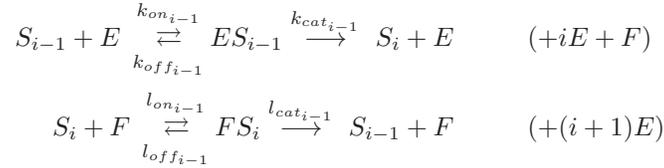
\[\begin{split} S_{i-1} + E \; \mathop{\stackrel{k_{on_{i-1}}}{\rightleftarrows}}_{k_{off_{i-1}}} ES_{i-1} \stackrel{k_{cat_{i-1}}}{\longrightarrow} \; S_i + E & \; \; \; \; \; \; \; \; \; \; (+iE + F) \\ S_i + F \; \mathop{\stackrel{l_{on_{i-1}}}{\rightleftarrows}}_{l_{off_{i-1}}} FS_i \stackrel{l_{cat_{i-1}}}{\longrightarrow} \; S_{i-1} + F & \; \; \; \; \; \; \; \; \; \; (+ (i+1)E) \end{split} \]
for all $i=1, \ldots, n$. This choice trivially satisfies condition $3.$ of the translation algorithm. The translated subnetworks $\tilde{\mathcal{N}}_i$ are given by
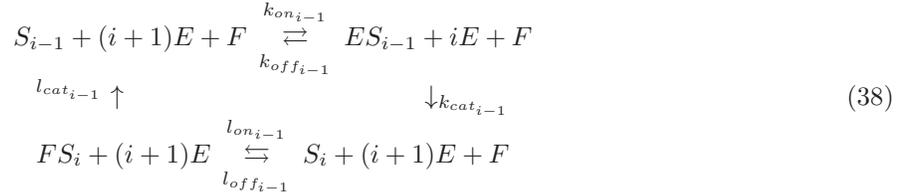
\begin{equation}
\label{system3}
\begin{array}{c} \displaystyle{S_{i-1}+(i+1)E+F \; \mathop{\stackrel{k_{on_{i-1}}}{\rightleftarrows}}_{k_{off_{i-1}}} \; ES_{i-1}+iE+F } \\ {}^{l_{cat_{i-1}}} \uparrow \; \; \; \; \; \; \; \; \; \; \; \; \; \; \; \; \; \; \; \; \; \; \; \; \; \; \; \; \; \; \; \; \; \; \; \; \; \; \; \; \; \downarrow_{k_{cat_{i-1}}} \\ \displaystyle{FS_i + (i+1)E \; \mathop{\stackrel{l_{on_{i-1}}}{\leftrightarrows}}_{l_{off_{i-1}}} \; S_i + (i+1)E + F} \end{array}
\end{equation}
where, for each $\tilde{\mathcal{N}}_i$, we associate the kinetic complexes $S_{i-1} + E$, $ES_{i-1}$, $S_i + F$, and $FS_i$, respectively, to the translated complexes in (\ref{system3}), starting in the top left and rotating clockwise. This choice satisfies step $4.$ of the translation algorithm.

We notice importantly that each translated subnetwork $\mathcal{N}_i$, $i=1,\ldots, n$, has a stoichiometrically distinct set of translated complexes and consequently that the translation is proper. (Notice that this would not have been satisfied if we had chosen the simpler translations $(+E)$ and $(+F)$ as in (\ref{futilecycletranslation}), although the choice of the additional factor of $iE$ to each subnetwork was arbitrary.) Consequently, for every mass action system $\mathcal{M} = (\mathcal{S},\mathcal{C},\mathcal{R},k)$ corresponding to $\mathcal{N}$ we may define the translated mass action system $\tilde{\mathcal{M}} = (\mathcal{S},\tilde{\mathcal{C}},\mathcal{CR},\tilde{\mathcal{R}},\tilde{k})$ according to Definition \ref{properkinetic}, i.e. taking $\tilde{k}_{h_2(i)} = k_i$.

In order to apply Theorem \ref{theorem01}, we need to compute $\tilde{\delta}$ and $\tilde{\delta}_K$ for the translation $\tilde{\mathcal{N}} = \bigcup_{i=1}^n \tilde{\mathcal{N}}_i$. We notice quickly that the translated network $\tilde{\mathcal{N}}$ has $n$ linkage classes and $4n$ distinct complexes so we only need to compute the stoichiometric space; however, this is the same as the original network. Since there are $3(n+1)$ species and $3$ conservation laws, we have that dim$(S) = 3n$. It follows that $\tilde{\delta} = 4n - 3n - n = 0$. A similar argument shows that $\tilde{\delta}=0$. Since the translation is strong, it follows by claim $1.$ of Theorem \ref{theorem01} that $\mathcal{M}$ has toric steady states for all rate constant vectors $k \in \mathbb{R}_{>0}^r$.

To characterize the steady state set of $\mathcal{M}$, we decompose the translated kinetic matrix $\tilde{A}_k$ of $\mathcal{M}$ into the block diagonal form
\[\tilde{A}_k = \left[ \begin{array}{cccc} (\tilde{A}_k)_1 & 0 & \cdots & 0 \\ 0 & (\tilde{A}_k)_2 & \cdots & 0 \\ \vdots & \vdots & \ddots & \vdots \\ 0 & 0 & \cdots & (\tilde{A}_k)_n \end{array} \right]\]
where each block $(\tilde{A}_k)_i$ has the form
\[(\tilde{A}_k)_i = \left[ \begin{array}{cccc} -k_{on_{i-1}} & k_{off_{i-1}} & 0 & l_{cat_{i-1}} \\ k_{on_{i-1}} & -k_{off_{i-1}}-k_{cat_{i-1}} & 0 & 0 \\ 0 & k_{cat_{i-1}} & -l_{on_{i-1}} & l_{off_{i-1}} \\ 0 & 0 & l_{on_{i-1}} & -l_{off_{i-1}}-l_{cat_{i-1}} \end{array} \right].\]
This is structurally identical to (\ref{33}) and, consequently, the $i^{th}$ support vector of ker$(\tilde{A}_k)$,\\ $\tilde{\mathbf{K}}_i = (\cdots \; | \; (\tilde{K}_i)_1,(\tilde{K}_i)_2,(\tilde{K}_i)_3,(\tilde{K}_i)_{4} \; | \; \cdots)$, corresponding to the support $\Lambda_i$ of $\tilde{\mathcal{N}}_i$, has entries
\[\begin{split} (\tilde{K}_i)_1 & = (k_{off_{i-1}} + k_{cat_{i-1}})l_{on_{i-1}}l_{cat_{i-1}} \\ (\tilde{K}_i)_2 & = k_{on_{i-1}}l_{on_{i-1}}l_{cat_{i-1}} \\ (\tilde{K}_i)_3 & = k_{on_{i-1}}k_{cat_{i-1}}(l_{off_{i-1}}+l_{cat_{i-1}}) \\ (\tilde{K}_i)_4 & = k_{on_{i-1}}k_{cat_{i-1}}l_{on_{i-1}}\end{split}\]
corresponding to the tree constants (\ref{treeconstant}) of $\tilde{\mathcal{N}}_i$. It follows by claims $2.$ of Theorem \ref{theorem01} that the steady state set of $\mathcal{M}$ is generated by the binomials
\begin{equation}
\label{43}
\begin{split} & (\tilde{K}_i)_2 x_{S_{i-1}} x_{E} - (\tilde{K}_i)_1 x_{ES_{i-1}}, (\tilde{K}_i)_2 x_{S_i}x_{F} - (\tilde{K}_i)_3 x_{ES_{i-1}}, \\
& \hspace{1in} (\tilde{K}_i)_2 x_{FS_i} - (\tilde{K}_i)_4 x_{ES_{i-1}}\end{split}
\end{equation}
for $i=1, \ldots, n$. With a little work, this can be shown to give the parametrization contained derived in \cite{M-D-S-C} and the calculated by hand in \cite{W-S}. It can also be directly manipulated to given the \emph{steady state invariants} (3) and (5) in \cite{Gun}.

We note that the approach taken here provides significant new insight into the mechanism of the multiple futile cycle. We now know that ker$(\Sigma)$ for $\mathcal{M}$ is partitioned as in \cite{M-D-S-C} because these partitions correspond to the linkage classes of a proper translation $\tilde{\mathcal{N}}$ of $\mathcal{N}$. The translation $\tilde{\mathcal{N}}$ allows an easy computation of the coefficients in (\ref{43}) based on tree constants (\ref{treeconstant}). This is preferable to the lengthy computations made in \cite{M-D-S-C} and \cite{W-S}. We defer consideration of claims $4.$ and $5.$ of Theorem \ref{theorem01} to future work.



\subsection{Example III: Feinberg-Shinar example}

Consider the phosphorylation network $\mathcal{N}$ given by
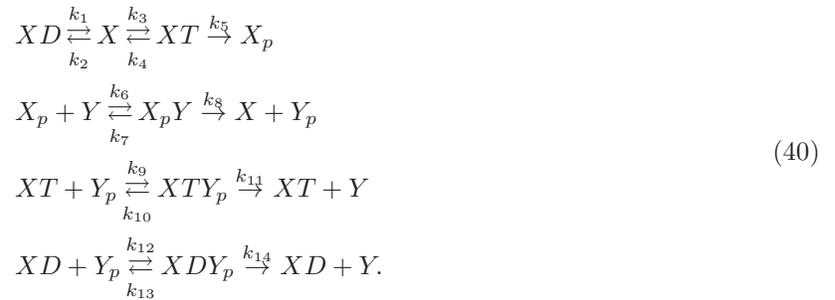
\begin{equation}
\label{system3}
\begin{split}
& \displaystyle{XD \mathop{\stackrel{k_1}{\rightleftarrows}}_{k_2} X \mathop{\stackrel{k_3}{\rightleftarrows}}_{k_4} XT \stackrel{k_5}{\rightarrow} X_p} \\
& \displaystyle{X_p + Y \mathop{\stackrel{k_6}{\rightleftarrows}}_{k_7} X_pY \stackrel{k_8}{\rightarrow} X + Y_p} \\
& \displaystyle{XT + Y_p \mathop{\stackrel{k_9}{\rightleftarrows}}_{k_{10}} XTY_p \stackrel{k_{11}}{\rightarrow} XT + Y} \\
& \displaystyle{XD + Y_p \mathop{\stackrel{k_{12}}{\rightleftarrows}}_{k_{13}} XDY_p \stackrel{k_{14}}{\rightarrow} XD +Y.}
\end{split}
\end{equation}
This network was first considered in Example (S60) of the Supporting Online Material of \cite{Sh-F}. The mass action systems $\mathcal{M}$ associated with $\mathcal{N}$ were shown in that paper to exhibit structural robustness at steady state with respect to concentrations of $[X_p]$. That is to say, the steady state sets was shown to be independent of $[X_p]$. The network was reproduced in \cite{M-D-S-C} where the authors showed that the systems $\mathcal{M}$ have toric steady states for all rate constant values.

We now show that the steady states of $\mathcal{M}$ can be characterized by appealing to network translation and Theorem \ref{theorem01}. We start by relabeling the species as in \cite{M-D-S-C}:
\[\mathcal{A}_1 = XD, \; \mathcal{A}_2 = X, \; \mathcal{A}_3 = XT, \; \mathcal{A}_4 = X_p, \; \mathcal{A}_5 = Y,\]
\[\mathcal{A}_6 = X_pY, \; \mathcal{A}_7 = Y_p, \; \mathcal{A}_8 = XTY_p, \; \mathcal{A}_9 = XDY_p\]
and assigning the complexes the indices
\[\mathcal{C}_1 = \mathcal{A}_1, \; \mathcal{C}_2 = \mathcal{A}_2, \; \mathcal{C}_3 = \mathcal{A}_3, \; \mathcal{C}_4 = \mathcal{A}_4, \; \mathcal{C}_5 = \mathcal{A}_4 + \mathcal{A}_5,\]
\[\mathcal{C}_6 = \mathcal{A}_6, \; \mathcal{C}_7 = \mathcal{A}_2 + \mathcal{A}_7, \; \mathcal{C}_8 = \mathcal{A}_3 + \mathcal{A}_7, \; \mathcal{C}_9 = \mathcal{A}_8,\]
\[\mathcal{C}_{10} = \mathcal{A}_3 + \mathcal{A}_5, \; \mathcal{C}_{11} = \mathcal{A}_1 + \mathcal{A}_7, \; \mathcal{C}_{12} = \mathcal{A}_9, \; \mathcal{C}_{13} = \mathcal{A}_1 + \mathcal{A}_5.\]

In order to determine a translation $\tilde{\mathcal{N}}$, we apply step $1.$ of the translation algorithm. There are two stoichiometric generators of ker$(Y \; I_a) \cap \mathbb{R}_{\geq 0}^r$:
\begin{equation}
\label{134}
\begin{split} E_1 & = (0,0,1,0,1,1,0,1,1,0,1,0,0,0) \\ E_2 & = (0,0,1,0,1,1,0,1,0,0,0,1,0,1).\end{split}
\end{equation}

\noindent We assign the orderings $\left\{3, 5, 6, 8, 9, 11 \right\}$ to $E_1$, $\left\{ 3, 5, 6, 8, 12, 14 \right\}$ to $E_2$, and translate each linkage class in the following way:
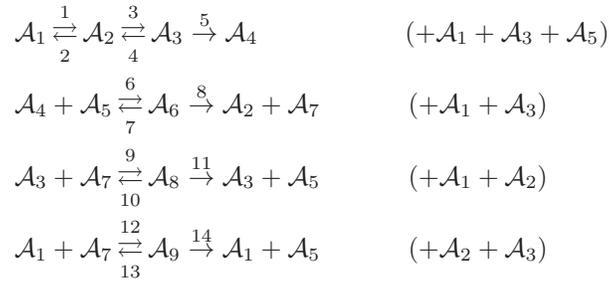
\[\begin{split}
& \displaystyle{\mathcal{A}_1 \mathop{\stackrel{1}{\rightleftarrows}}_{2} \mathcal{A}_2 \mathop{\stackrel{3}{\rightleftarrows}}_{4} \mathcal{A}_3 \stackrel{5}{\rightarrow} \mathcal{A}_4} \; \; \; \; \; \; \; \; \; \; \; \; \; \; \; \; \; \; \; \; (+ \mathcal{A}_1 + \mathcal{A}_3 + \mathcal{A}_5) \\
& \displaystyle{\mathcal{A}_4 + \mathcal{A}_5 \mathop{\stackrel{6}{\rightleftarrows}}_{7} \mathcal{A}_6 \stackrel{8}{\rightarrow} \mathcal{A}_2 + \mathcal{A}_7} \; \; \; \; \; \; \; \; \; \; \; \; (+ \mathcal{A}_1 + \mathcal{A}_3) \\
& \displaystyle{\mathcal{A}_3 + \mathcal{A}_7 \mathop{\stackrel{9}{\rightleftarrows}}_{10} \mathcal{A}_8 \stackrel{11}{\rightarrow} \mathcal{A}_3 + \mathcal{A}_5} \; \; \; \; \; \; \; \; \; \; \; \; (+ \mathcal{A}_1 + \mathcal{A}_2)\\
& \displaystyle{\mathcal{A}_1 + \mathcal{A}_7 \mathop{\stackrel{12}{\rightleftarrows}}_{13} \mathcal{A}_9 \stackrel{14}{\rightarrow} \mathcal{A}_1 + \mathcal{A}_5 \; \; \; \; \; \; \; \; \; \; \; \; (+ \mathcal{A}_2 + \mathcal{A}_3)}
\end{split}\]
This satisfies the requirements of step $2.$ and $3.$ of the translation algorithm and gives the follow translation $\tilde{\mathcal{N}}$, where we have labelled the reactions as they correspond to (\ref{system3}):
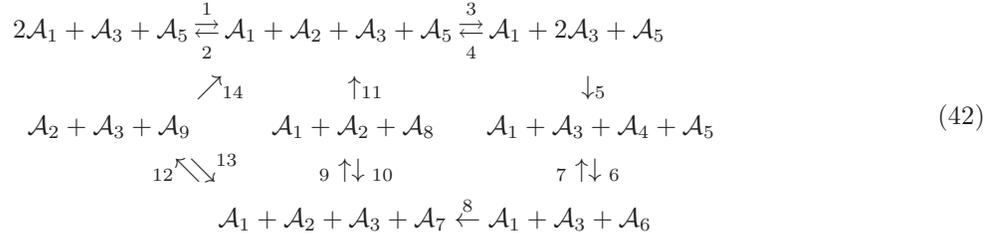
\begin{equation}
\label{system4}
\begin{split} & 2\mathcal{A}_1 + \mathcal{A}_3 + \mathcal{A}_5 \mathop{\stackrel{1}{\rightleftarrows}}_{2} \mathcal{A}_1 + \mathcal{A}_2 + \mathcal{A}_3 + \mathcal{A}_5 \mathop{\stackrel{3}{\rightleftarrows}}_{4} \mathcal{A}_1 + 2\mathcal{A}_3 + \mathcal{A}_5 \\ & \; \; \; \; \; \; \; \; \; \; \; \; \; \; \; \; \; \; \; \; \; \; \; \; \nearrow \hspace{-0.1cm} {}_{14} \; \; \; \; \; \; \; \; \; \; \; \; \; \uparrow_{11} \; \; \; \; \; \; \; \; \; \; \; \; \; \; \; \; \; \; \; \; \; \; \; \; \; \; \; \downarrow_{5} \\ & \; \; \mathcal{A}_2 + \mathcal{A}_3 + \mathcal{A}_9 \; \; \; \; \; \; \; \; \; \; \; \mathcal{A}_1 + \mathcal{A}_2 + \mathcal{A}_8 \; \; \; \; \; \; \; \mathcal{A}_1 + \mathcal{A}_3 + \mathcal{A}_4 + \mathcal{A}_5 \\ & \; \; \; \; \; \; \; \; \; \; \; \; \; \; \; \; \; \; \; {}_{12} \hspace{-0.1cm} \nwarrow \hspace{-0.15cm} \searrow \hspace{-0.1cm} {}^{13} \; \; \; \; \; \; \; \; \; \; \; {}_9 \uparrow \downarrow {}_{10} \; \; \; \; \; \; \; \; \; \; \; \; \; \; \; \; \; \; \; \; \; \; {}_7 \uparrow \downarrow {}_6 \\ & \; \; \; \; \; \; \; \; \; \; \; \; \; \; \; \; \; \; \; \; \; \; \; \; \; \; \; \; \mathcal{A}_1 + \mathcal{A}_2 + \mathcal{A}_3 + \mathcal{A}_7 \stackrel{8}{\leftarrow} \mathcal{A}_1 + \mathcal{A}_3 + \mathcal{A}_6 \end{split}
\end{equation}
We notice that the stoichiometric generators $E_1$ and $E_2$ in (\ref{134}) now correspond to cycles in this reaction graph. We associate the following indices to the translated complex set $\tilde{\mathcal{C}}$:
\[\tilde{\mathcal{C}}_1 = 2\mathcal{A}_1 + \mathcal{A}_3 + \mathcal{A}_5, \; \tilde{\mathcal{C}}_2 = \mathcal{A}_1 + \mathcal{A}_2 + \mathcal{A}_3 + \mathcal{A}_5, \; \tilde{\mathcal{C}}_3 = \mathcal{A}_1 + 2\mathcal{A}_3 + \mathcal{A}_5,\]
\[\tilde{\mathcal{C}}_4 = \mathcal{A}_1 + \mathcal{A}_3 + \mathcal{A}_4 + \mathcal{A}_5, \; \tilde{\mathcal{C}}_5 = \mathcal{A}_1 + \mathcal{A}_3 + \mathcal{A}_6, \; \tilde{\mathcal{C}}_6 = \mathcal{A}_1 + \mathcal{A}_2 + \mathcal{A}_3 + \mathcal{A}_7,\]
\[\tilde{\mathcal{C}}_7 = \mathcal{A}_1 + \mathcal{A}_2 + \mathcal{A}_8, \; \tilde{\mathcal{C}}_8 = \mathcal{A}_2 + \mathcal{A}_3 + \mathcal{A}_9.\]

When attempting to assign kinetic complexes to $\tilde{\mathcal{N}}$ by step $4.$ of the algorithm, we notice that we may not directly assign the pre-translation source complexes $\mathcal{CR}$ in $\mathcal{N}$ to the source complexes $\tilde{\mathcal{CR}}$ in $\tilde{\mathcal{N}}$ because the source complexes of $\mathcal{R}_9$ and $\mathcal{R}_{12}$ ($\mathcal{C}_8$ and $\mathcal{C}_{11}$, respectively) are \emph{both} translated to the complex $\tilde{\mathcal{C}}_{6}$ in $\tilde{\mathcal{N}}$. That is to say, $\tilde{\mathcal{N}}$ is an \emph{improper} translation of $\mathcal{N}$ with improper complex set $\tilde{\mathcal{C}}_I = \left\{ \tilde{\mathcal{C}}_6 \right\}$. By step $4.$ of the translation algorithm, we may choose either $\mathcal{C}_{8}$ or $\mathcal{C}_{11}$ to be the kinetic complex corresponding $\tilde{\mathcal{C}}_{6}$. There is no preference between the two, so we will arbitrarily choose $\mathcal{C}_{8}$. This leaves the improper reaction set $\mathcal{R}_I = \left\{ \mathcal{R}_{12} \right\}$. We choose the rest of the kinetic complexes in the natural way to complete the set $\mathcal{CR}_K = \left\{ \mathcal{C}_1,\mathcal{C}_2,\mathcal{C}_3,\mathcal{C}_5,\mathcal{C}_6,\mathcal{C}_8,\mathcal{C}_9,\mathcal{C}_{12} \right\}$, but we notice that $\mathcal{CR}_K \subset \mathcal{CR}$.

Since the translation $\tilde{\mathcal{N}}$ is improper, the dynamics (\ref{gde}) governing any generalized mass action system $\tilde{\mathcal{M}}$ corresponding to $\tilde{\mathcal{N}}$ will have fewer monomials than any system (\ref{de}) governing a mass action system $\mathcal{M}$ corresponding to $\mathcal{N}$. We may still, however, be able to relate $\tilde{\mathcal{M}}$ and $\mathcal{M}$ \emph{at steady state} by Theorem \ref{theorem01}. To apply this result, we must show that $\tilde{\mathcal{N}}$ is strongly resolvably improper and $\tilde{\delta}=\tilde{\delta}_K=0$.

We show firstly that it is weakly resolvably improper (Definition \ref{improperkineticsubspace}). We have that $\tilde{\mathcal{N}}$ is a strong translation because it is weakly reversible, so we only need to check $\tilde{S}_I \subset \tilde{S}$. In order to do so, we need to consider the space
\[\begin{split} \tilde{S}_I & = \mbox{span} \left\{ \bigcup_{i \in \mathcal{R}_I} \left\{ y_{\rho(i)} - y_{\rho(i)_K} \right\} \right\} = \mbox{span} \left\{ y_{11} - y_{8} \right\} \\ & = \mbox{span} \left\{ (1,0,0,0,0,0,1,0,0) - (0,0,1,0,0,0,1,0,0) \right\} \\ & = \mbox{span} \left\{ (1,0,-1,0,0,0,0,0,0) \right\}.\end{split}\]
Since $\tilde{\mathcal{N}}$ is weakly reversible and only has a single linkage class, the kinetic-order subspace $\tilde{S}$ is given by the span of the stoichiometric differences of the complexes in the set $\mathcal{CR}_K$. We can easily see that $y_1 - y_3 = (1,0,-1,0,0,0,0,0,0)$ so that $\tilde{S}_I \subset \tilde{S}$ and, consequently, $\tilde{\mathcal{N}}$ is a weakly resolvably improper translation of $\mathcal{N}$.

We now want to check whether $\tilde{\mathcal{N}}$ is \emph{strongly} resolvably improper. To do this, we need to determine the weak kinetic adjustment factors $\tilde{K}_{\rho(i),\rho(i)_K}(\mathbf{x})$ for all $i \in \mathcal{R}_I$. Since we have that $\mathcal{R}_I = \left\{ \mathcal{R}_{12} \right\}$, we may use the observation of the previous paragraph to write
\[y_{11} - y_{8} = y_1 - y_3 \; \; \; \Longrightarrow \; \; \; \mathbf{x}^{y_{11}} = \left( \frac{\mathbf{x}^{y_1}}{\mathbf{x}^{y_3}} \right) \mathbf{x}^{y_{8}}.\]
This is the form required of Lemma \ref{lemma00} so that the weak kinetic adjustment factor of $\mathcal{R}_{12}$ is
\begin{equation}
\label{24715}
\tilde{K}_{\rho(12),\rho(12)_K} = \left( \frac{\mathbf{x}^{y_1}}{\mathbf{x}^{y_3}} \right).
\end{equation}

To determine the form of the \emph{strong} kinetic adjustment factors (\ref{strongkineticadjustmentfactor}), we first need to construct the semi-proper reaction graph $\tilde{G}(\tilde{V},\tilde{E})$ of (\ref{system4}) by Definition \ref{hypothetical}. Since $\mathcal{R}_I = \left\{ \mathcal{R}_{12} \right\}$, we may assign $\tilde{k}_i = k_i$ for $i \not= 12$ and leave $\tilde{k}_{12}$ undetermined. By (\ref{24715}) we have that the strong kinetic adjustment factor of $\mathcal{R}_{12}$ is
\begin{equation}
\label{939999}
\tilde{K}_{\rho(12),\rho(12)_K} = \left( \frac{\tilde{K}_{h_2(1)}}{\tilde{K}_{h_2(3)}} \right) = \left( \frac{\tilde{K}_1}{\tilde{K}_3} \right)
\end{equation}
where $\tilde{K}_i$ is the tree constant (\ref{treeconstant}) for the $i^{th}$ complex in $\tilde{G}(\tilde{V},\tilde{E})$. To determine the tree constants (\ref{treeconstant}), we construct the relevant kinetic matrix $\tilde{A}_k$ for $\tilde{\mathcal{N}}$ with the rate constant choices for $\tilde{G}(\tilde{V},\tilde{E})$:
\footnotesize
\begin{equation}\tilde{A}_k = \left[ \begin{array}{cccccccc}
 \hspace{-0.2cm} -k_1 & k_2 & 0 & 0 & 0 & 0 & 0 & 0 \\
k_1 & \hspace{-0.2cm} -k_2-k_3 & k_4 & 0 & 0 & 0 & k_{11} & k_{14} \\
0 & k_3 & \hspace{-0.4cm} -k_4 -k_5 & 0 & 0 & 0 &0 &0 \\
0 & 0 & k_5 & \hspace{-0.3cm} -k_6 & k_7 & 0 & 0 & 0 \\
0 & 0 & 0 & k_6 & \hspace{-0.2cm} -k_7 -k_8 & 0 & 0 & 0 \\
0 & 0 & 0 & 0 & k_8 & \hspace{-0.4cm} -k_9 - k_{13} & k_{10} & \tilde{k}_{12} \\
0 & 0 & 0 & 0 & 0 & k_9 & \hspace{-0.4cm} -k_{10}-k_{11} & 0 \\
0 & 0 & 0 & 0 & 0 & k_{13} & 0 & \hspace{-0.4cm} -\tilde{k}_{12}-k_{14}
\end{array} \right]. \end{equation}
\small
The relevant tree constants are
\[\begin{split} \tilde{K}_1 & = k_2(k_4+k_5)k_6k_8 \left[ k_9k_{11}(\tilde{k}_{12}+k_{14})+ (k_{10}+k_{11})k_{13}k_{14} \right] \\ \tilde{K}_3 & = k_1k_3k_6k_8\left[ k_9k_{11}(\tilde{k}_{12}+k_{14})+ (k_{10}+k_{11})k_{13}k_{14} \right].\end{split}\]
It follows from (\ref{939999}) that we have the strong kinetic adjustment factor simplifies to
\[\tilde{K}_{\rho(12),\rho(12)_K} = \left( \frac{\tilde{K}_1}{\tilde{K}_3} \right) = \frac{k_2(k_4+k_5)}{k_1k_3}.\]
Since this does not depend on any rate constant corresponding to a reaction in $\mathcal{R}_I$ (i.e. the rate constant $\tilde{k}_{12}$), it follows that $\tilde{\mathcal{N}}$ is a strongly resolvable improper translation of $\mathcal{N}$ by Definition \ref{stronglyresolvable}.

We are now prepared to define the improperly translated mass action system $\tilde{\mathcal{M}} = (\mathcal{S},\tilde{\mathcal{C}},\mathcal{CR}_K,\tilde{\mathcal{R}},\tilde{k})$ associated with the translation $\tilde{\mathcal{N}}$. By Definition \ref{improperkinetic}, we assign $\tilde{k}_i = k_i$ for all $i \not= 12$ (as in the semi-proper reaction graph) and define the rate constant corresponding to $\mathcal{R}_{12}$ to be
\begin{equation}
 \label{3100}
\tilde{k}_{12} = \left( \tilde{K}_{\rho(12),\rho(12)_K}\right) k_{12} = \left( \frac{k_2(k_4+k_5)}{k_1k_3} \right) k_{12}.
\end{equation}
This defines the improperly translated mass action system $\tilde{\mathcal{M}}$. Since $\tilde{\delta}=0$, we have by Lemma \ref{lemma331} that the system (\ref{de}) governing $\mathcal{M}$ and the system (\ref{gde}) governing $\tilde{\mathcal{M}}$ have the same steady state set.

We now seek to apply Theorem \ref{theorem01} to characterize the steady state set of $\mathcal{M}$. Since $\tilde{\mathcal{N}}$ is strongly resolvably improper translation of $\mathcal{N}$, and since $\tilde{\delta}=\tilde{\delta}_K=0$ (easily checked), we have by claim $1.$ of Theorem \ref{theorem01} that $\mathcal{M}$ has toric steady states for all rate constant values. In order to apply claim $2.$, we need to compute the tree constants $\tilde{K}_i$, $i=1,\ldots,8$, corresponding to the reaction graph of $\tilde{\mathcal{M}}$. By (\ref{3100}), we have that
\footnotesize
\[\begin{split}
\tilde{K}_1 & = k_2(k_4+k_5)k_6k_8 \left[ k_9k_{11}\left(\left( \frac{k_2(k_4+k_5)}{k_1k_3} \right) k_{12}+k_{14} \right)+ \left(k_{10}+k_{11} \right)k_{13}k_{14} \right]\\
\tilde{K}_2 & = k_1(k_4+k_5)k_6k_8 \left[ k_9k_{11}\left(\left( \frac{k_2(k_4+k_5)}{k_1k_3} \right) k_{12}+k_{14} \right)+ \left(k_{10}+k_{11} \right)k_{13}k_{14} \right]\\
\tilde{K}_3 & = k_1k_3k_6k_8\left[ k_9k_{11} \left( \left( \frac{k_2(k_4+k_5)}{k_1k_3} \right) k_{12}+k_{14} \right)+ \left( k_{10}+k_{11}\right)k_{13}k_{14} \right]\\
\tilde{K}_4 & = k_1k_3k_5(k_7+k_8)\left[ k_9k_{11} \left( \left( \frac{k_2(k_4+k_5)}{k_1k_3} \right) k_{12}+k_{14} \right)+ \left( k_{10}+k_{11}\right)k_{13}k_{14} \right]\\
\tilde{K}_5 & = k_1k_3k_5k_6\left[ k_9k_{11} \left( \left( \frac{k_2(k_4+k_5)}{k_1k_3} \right) k_{12}+k_{14} \right)+ \left( k_{10}+k_{11}\right)k_{13}k_{14} \right]\\
\tilde{K}_6 & = k_1k_3k_5k_6k_8(k_{10}+k_{11})\left( \left( \frac{k_2(k_4+k_5)}{k_1k_3} \right)k_{12}+ k_{14}\right)\\
\tilde{K}_7 & = k_1k_3k_5k_6k_8k_9 \left( \left( \frac{k_2(k_4+k_5)}{k_1k_3} \right)k_{12}+ k_{14}\right)\\
\tilde{K}_8 & = k_1k_3k_5k_6k_8\left( k_{10}+k_{11} \right) k_{13}
\end{split}\]
\small
By claim $2.$ of Theorem \ref{theorem01}, therefore, we have that the steady state set of $\mathcal{M}$ can be generated by the binomials
\[\tilde{K}_{3} x_3x_7 - \tilde{K}_{6} x_3, \tilde{K}_{3} x_9 - \tilde{K}_{8} x_3, \tilde{K}_{3} x_8 - \tilde{K}_{7} x_3, \tilde{K}_{3} x_6 - \tilde{K}_{5} x_3,\]
\[\tilde{K}_{3} x_4x_5 - \tilde{K}_{4} x_3, \tilde{K}_{3} x_2 - \tilde{K}_{2} x_3, \tilde{K}_{3} x_1 - \tilde{K}_{1} x_3.\]
After simplification, this can be seen to be the same as the binomials given by (3.16) in \cite{M-D-S-C}. In other words, we have found the Groebner basis of the steady state ideal with respect to the ordering $x_1>x_2>x_4>x_5>x_6>x_8>x_9>x_3>x_7$. Of particular importance, we have related the monomial coefficients explicitly to a reaction graph. It was not, however, the reaction graph of the chemical reaction network $\mathcal{N}$; rather it was the reaction graph of the improperly translated chemical reaction network $\tilde{\mathcal{N}}$ with the rate constants given by the corresponding improperly translated mass action system $\tilde{\mathcal{M}}$. Notable, we were not permitted to define $\tilde{k}_{12}=k_{12}$; rather, we had to define it by (\ref{3100}). We once again defer consideration of claims $4.$ and $5.$ of Theorem \ref{theorem01} to future work.

\section{Conclusions and Future Work}

In this paper, we introduced the notion of a \emph{translated chemical reaction network} as a method for characterizing the steady states of mass action systems.


The method of network translation relates a chemical reaction network $\mathcal{N} = (\mathcal{S},\mathcal{C},\mathcal{R})$ to a generalized chemical reaction network $\tilde{\mathcal{N}} = (\mathcal{S},\tilde{\mathcal{C}},\mathcal{CR}_K, \tilde{\mathcal{R}})$, called a \emph{translation} of $\mathcal{N}$, which has the same reaction vectors as $\mathcal{N}$ but different complexes and consequently different connectivity properties in the translated reaction graph. We defined two classes of translations, \emph{proper} translations (Definition \ref{supp}) and \emph{strongly resolvably improper} translations (Definition \ref{stronglyresolvable}), which allowed a translated mass action system $\tilde{\mathcal{M}} = (\mathcal{S},\tilde{\mathcal{C}},\mathcal{CR}_K, \tilde{\mathcal{R}},\tilde{k})$ to be defined (Definition \ref{properkinetic} and Definition \ref{improperkinetic}, respectively). We then presented conditions on the network topology of $\tilde{\mathcal{N}}$ which allowed an explicit connection to be made between complex balanced steady states of $\tilde{\mathcal{M}}$ and toric steady states of $\mathcal{M}$ (Theorem \ref{theorem01}). Finally, in Section \ref{examplesection}, we applied the results to a series of examples drawn from the literature.

The study of translated chemical reaction networks specifically, and generalized chemical reaction networks in general, is very new and there are consequently many aspects of the theory which have not be fully investigated. A few of the key points of future work include:
\begin{enumerate}
\item
The translation algorithm presented in Section \ref{techniquesection} depends heavily on intuition which may be lacking for large-scale biochemical networks. A stronger algorithm, and computational implementation, is required for broad-based application.
\item
There is notable room for improvement in the conditions for \emph{weak} and \emph{strong} resolvability of improper translations (Definitions \ref{improperkineticsubspace} and \ref{stronglyresolvable}). In particular, it is undesirable to construct the semi-proper reaction graph and compute all the ratios $\tilde{K}_{h_2(p_i)} / \tilde{K}_{h_2(q_i)}$ in order to determine strong resolvability. The author suspects that there are simpler sufficient conditions for strong resolvability.
\item
Translated chemical reaction networks are generalized chemical reaction networks, and consequently conclusions may only be drawn as far as they are justified by this underlying theory. The author suspects that, as this nascent theory becomes more fully developed, there will be increased application for the process of network translations in characterizing the steady states of mass action systems.
\end{enumerate}

\noindent \textbf{Acknowledgements:} The author is grateful for the numerous constructive conversations with Anne Shiu, Carsten Conradi, Casian Pantea, Stefen M\"{u}ller, and others, over email and at the AIM workshop ``Mathematical problems arising from biochemical reaction networks,'' which pointed him toward the strong connection between toric steady states and complex balancing in generalized mass action systems.

\appendix

\section{Appendix (Deficiency Result)} \label{AppendixB}
\begin{lemma}
\label{deficiencylemma}
The deficiency $\delta = \mbox{\emph{dim}}(\mbox{\emph{ker}}(Y) \cap \mbox{\emph{Im}}(I_a))$ of a chemical reaction network $\mathcal{N}$ is equivalent to $\delta = n - \ell - s$ where $n$ is the number of stoichiometrically distinct complexes, $\ell$ is the number of linkage classes, and $s =$dim$(S)$.
\end{lemma}
\begin{proof}
It follows from basic dimensional considerations that
\[\mbox{dim}(\mbox{ker}(\Gamma)) = \mbox{dim}(\mbox{ker} (I_a)) + \mbox{dim}(\mbox{ker}(Y) \cap \mbox{Im}(I_a)).\]
From the rank-nullity theorem, we may relate the dimension of the kernel of a matrix to its rank. Consequently, we have
\[\mbox{dim}(\mbox{ker}(\Gamma)) = r - \mbox{dim}(\mbox{Im}(\Gamma)) = r-s.\]
The rank of $I_a$ corresponds to the number of complexes minus the number of linkage classes, so that dim(Im$(I_a)) = n- \ell$. It follows that
\[\mbox{dim}(\mbox{ker}(I_a)) = r - (n - \ell) = r + \ell - n.\]
It follows that
\[\delta = \mbox{dim}(\mbox{ker}(Y) \cap \mbox{Im}(I_a)) = \mbox{dim}(\mbox{ker}(\Gamma)) - \mbox{dim}(\mbox{ker} (I_a))\]
\[ = (r-s) - (r+\ell-n) = n - \ell - s\]
and we are done.
\end{proof}

\section{Appendix (Kernel of $A_k$)} \label{AppendixC}

In this appendix, we present the a more detailed characterization of ker$(A_k)$ for a mass action system $\mathcal{M} = (\mathcal{S},\mathcal{C},\mathcal{R},k)$.

Consider a weakly reversible chemical reaction network $\mathcal{N} =(\mathcal{S},\mathcal{C},\mathcal{R})$ and let $\Lambda_k$, $k=1, \ldots, \ell$, denote the supports of the network's linkage classes $\mathcal{L}_k$, $k=1, \ldots, \ell$. Define a subgraph $\mathcal{T} \subset \mathcal{R}$ to be a \emph{spanning} $i$-\emph{tree} if $\mathcal{T}$ spans all of the complexes in some linkage class $\mathcal{L}_k$, contains no cycles, and has the unique sink $\mathcal{C}_i \in \mathcal{C}$. Let $\mathcal{T}_i$ denote the set of all spanning $i$-trees for $\mathcal{C}_i \in \mathcal{C}$. We define the following network constants.
\begin{definition}
Consider a weakly reversible chemical reaction network $\mathcal{N} = (\mathcal{S},\mathcal{C},\mathcal{R})$ with reaction weightings $k_j$, $j=1, \ldots, r$. Then the \textbf{tree constant} of $\mathcal{C}_i \in \mathcal{C}$ is given by
\begin{equation}
\label{treeconstant}
K_i = \sum_{\mathcal{T} \in \mathcal{T}_i} \prod_{\mathcal{R}_j \in \mathcal{T}_i} k_j.
\end{equation}
\end{definition}
\begin{remark} To compute the tree constants $K_i$, we restrict ourselves to the linkage class containing the complex $\mathcal{C}_i \in \mathcal{C}$. We then determine all of the spanning trees which contain $\mathcal{C}_i$ as the unique sink, multiply across all the weighted edges in each tree, and then sum over all such trees. The terms $K_i$ can also be computed by computing specific minors of the kinetic matrix $A_k$ restricted to the support of the linkage classes (Proposition 3, \cite{C-D-S-S}). Note that the term ``tree constant'' is our own.
\end{remark}

The following result characterizes ker$(A_k)$ in terms of the tree constants (\ref{treeconstant}). This result appears in various forms within the chemical reaction network literature. A basic form, just concerned with the signs of the individual components, can be found in \cite{F3} (Proposition 4.1) and \cite{G-H} (Theorem 3.1). A more specific result can be obtained by the \emph{Matrix-Tree Theorem} \cite{Stanley}. This form is explicitly connected with the reaction graph of a chemical reaction network in \cite{C-D-S-S} (Corollary 4). A direct argument is also contained in Section 3.4 of \cite{J}. We defer to these references for the proof.

\begin{theorem}
\label{Akernel}
Let $\mathcal{N} = (\mathcal{S},\mathcal{C},\mathcal{R})$ denote a weakly reversible chemical reaction network. Let $\Lambda_k$, $k=1, \ldots, \ell$, denote the supports of the network's linkage classes $\mathcal{L}_k$, $k=1, \ldots, \ell$, and let $K_i$ denote the tree constants (\ref{treeconstant}) corresponding to the complexes $\mathcal{C}_i \in \mathcal{C}$. Then
\[\mbox{ker} (A_k) = \mbox{span} \left\{ \mathbf{K}_1, \mathbf{K}_2, \ldots, \mathbf{K}_{\ell} \right\}\]
where $\mathbf{K}_j = ([K_j]_1,[K_j]_2,\ldots,[K_j]_n)$ has entries
\[ [K_j]_i = \left\{ \begin{array}{ll} K_i, \; \; \; \; \; \; & \mbox{if } i \in \Lambda_j \\ 0 & \mbox{otherwise.} \end{array} \right.\]
\end{theorem}

\begin{remark}
This theorem may be extended to networks which are not weakly reversible by considering the terminal strongly linked components of a chemical reaction network. As all the relevant networks considered in this paper are weakly reversible, however, Theorem \ref{Akernel} will suffice for our purposes here.
\end{remark}





\end{document}